\documentclass[11pt]{article}
\usepackage[utf8]{inputenc}
\usepackage{amsfonts, amsmath}
\usepackage{amssymb}
\usepackage{amsthm,amsmath,amscd}
\usepackage{enumerate}
\usepackage{url}
\usepackage{amsbsy}
\usepackage[pdftex]{graphicx}
\usepackage[margin=25pt,font=small,labelfont=bf]{caption}
\usepackage[breaklinks,hyperfootnotes=false,bookmarks=false,colorlinks=true]{hyperref}
\usepackage{subfig}
\usepackage{mdframed}
\usepackage{color}
\usepackage[square,numbers,sort&compress]{natbib}
\newcommand*{\doi}[1]{doi: \href{https://dx.doi.org/#1}{\urlstyle{rm}\nolinkurl{#1}}}
\newcommand*{\arxiv}[1]{arXiv:  \href{https://arxiv.org/abs/#1}{\urlstyle{rm}\nolinkurl{#1}}}

\usepackage [autostyle, english = american]{csquotes}
\MakeOuterQuote{"}

\newcommand{\defn}[1]{\textcolor{blue}{\emph{#1}}}

\graphicspath{{figs/}}

\newcommand{\RR}{\mathbb R}

\newcommand{\bna}{\begin{eqnarray}}
\newcommand{\ena}{\end{eqnarray}}
\newcommand{\ba}{\begin{eqnarray*}}
\newcommand{\ea}{\end{eqnarray*}}
\newcommand{\bs}[1]{}

\newtheorem{theorem}{Theorem}[section]
\newtheorem{corollary}[theorem]{Corollary}
\newtheorem{lemma}[theorem]{Lemma}
\newtheorem{proposition}[theorem]{Proposition}
\newtheorem{remark}[theorem]{Remark}
\newtheorem{definition}[theorem]{Definition}

\newcommand{\CC}{{\mathbb C}}

\def\p{{\bf p}}

\def\f{{\bf f}}

\def\v{{\bf v}}

\def\r{{\bf r}}

\def\s{{\bf s}}

\title{Packing Disks by Flipping and Flowing}
\author{Robert Connelly\thanks{Department of Mathematics, Cornell University. \texttt{rc46@cornell.edu}. Partially supported by 
NSF grant DMS-1564493.}
\and Steven J. Gortler\thanks{School of Engineering and Applied Sciences, Harvard University. \texttt{sjg@cs.harvard.edu}. 
Partially supported by NSF grant DMS-1564473.}
}
\date{}

\date{}

\usepackage[margin=1in]{geometry}
\begin{document}
\maketitle 
\begin{abstract}
We provide a new type of proof for the Koebe-Andreev-Thurston (KAT) 
planar circle packing theorem based on combinatorial edge-flips.
In particular, we show that starting from a disk packing with
a maximal planar contact graph $G$, one can remove any
flippable edge $e^-$ of this graph
and then continuously flow the disks
in the plane, such that at the end of the flow, 
one obtains a new disk packing whose contact graph 
is the graph resulting from flipping the edge $e^-$ in $G$.
This flow is parameterized by a single inversive distance.
\end{abstract}

\section{Introduction}

The well known Koebe–Andreev–Thurston (KAT) 
(planar) circle packing theorem states that
for every 
planar graph $G$ with $n$ vertices,
there is a corresponding packing of $n$ disks 
(with mutually disjoint interiors)
in
the plane, whose
contact graph is isomorphic to $G$. 
Moreover if $G$ is maximal (ie. all faces triangular), then
this packing is
unique up to M{\"o}bius transformations and reflections~\cite{kobe,andreev,thurston}.
This theorem is important in conformal geometry~\cite{rodin} and
has been generalized in numerous directions~\cite{thurston, rivin, schramm}.  In \cite{LFT}, a packing 
whose graph is a triangulation of the plane is called \emph{a compact packing}.

There are a variety of techniques that have been used to establish
the KAT theorem, 
These include methods based on 
conformal geometry~\cite{kobe}, 
combinatorial analysis of hyperbolic polyhedra~\cite{andreev, roeder}
circle geometry, cone-singularities and topology of continuous maps~\cite{thurston,marden}, 
variational methods~\cite{cdv,rivin,bobenko}, iterative "flattening" algorithms~\cite{mohar,ken,kenbook}, 
and combinatorial Ricci Flow~\cite{chow}.

In this paper, we present a novel simple and
constructive approach for proving the 
KAT circle packing theorem based on ideas from (local) rigidity theory~\cite{bbp,sticky}.

We start with any 
"trilaterated" packing;  such a packing is constructed by starting 
with three fixed disks in mutual tangential contact, and then adding in, one-by-one, new disks that are tangential
to three previously placed disks.
It will be easy to establish that 
such a packing is unique up to M{\"o}bius transformation and reflection. (See Figure~\ref{fig:canon}.)

Next, we use the fact that one can always 
combinatorially transform 
any maximal planar graph   to any desired
target maximal planar graph $G$
using a finite sequence of combinatorial 
"edge flips" ~\cite{wagner,bose,sleator}. (See Figure~\ref{fig:flip}.)

Then, at the heart of this paper, we show how we can continuously 
flow the planar packing in coordination with
a single edge flip. More specifically,
let $P(0)$ be a packing with a maximal planar contact graph $G$.
Let $H$ be a maximal planar graph obtained from  $G$ by removing one edge $e^-$ and adding its "cross" edge
$e^+$ (see Figure~\ref{fig:flip}). Also let $G^-$ be the intermediate
graph $G\setminus e^-$ such that $G^-$ has one quadrilateral 
face. Then we can always find a continuous path of planar
packings $P(t)$ for $t \in [0..1]$ with the following properties:
\begin{itemize}
    \item The contact graph of $P(0)$ is $G$.
    \item For $t$ in the interval $(0,1)$, the contact graph of $P(t)$ is $G^-$.
    \item The contact graph of $P(1)$ is $H$.
\end{itemize}
This means that we can continuously push the two disks 
corresponding to $e^-$ apart,
while maintaining a packing (no interior overlaps) and maintaining 
all of the other contacts in $G^-$, and we can
keep doing this until the two disks across $e^+$ come together (see Figure~\ref{fig:flow}), giving us $H$ as the contact graph.
Moreover, up to the speed of the parameterization and  up to M{\"o}bius transformations and reflections,
this path is unique and is parameterized by a single inversive distance.
The construction of this path is our main technical result, Theorem~\ref{thm:flow}.

  \begin{figure}[t]
    \centering
    \begin{tabular}{cccc}
  {     \includegraphics[width=0.23\textwidth]{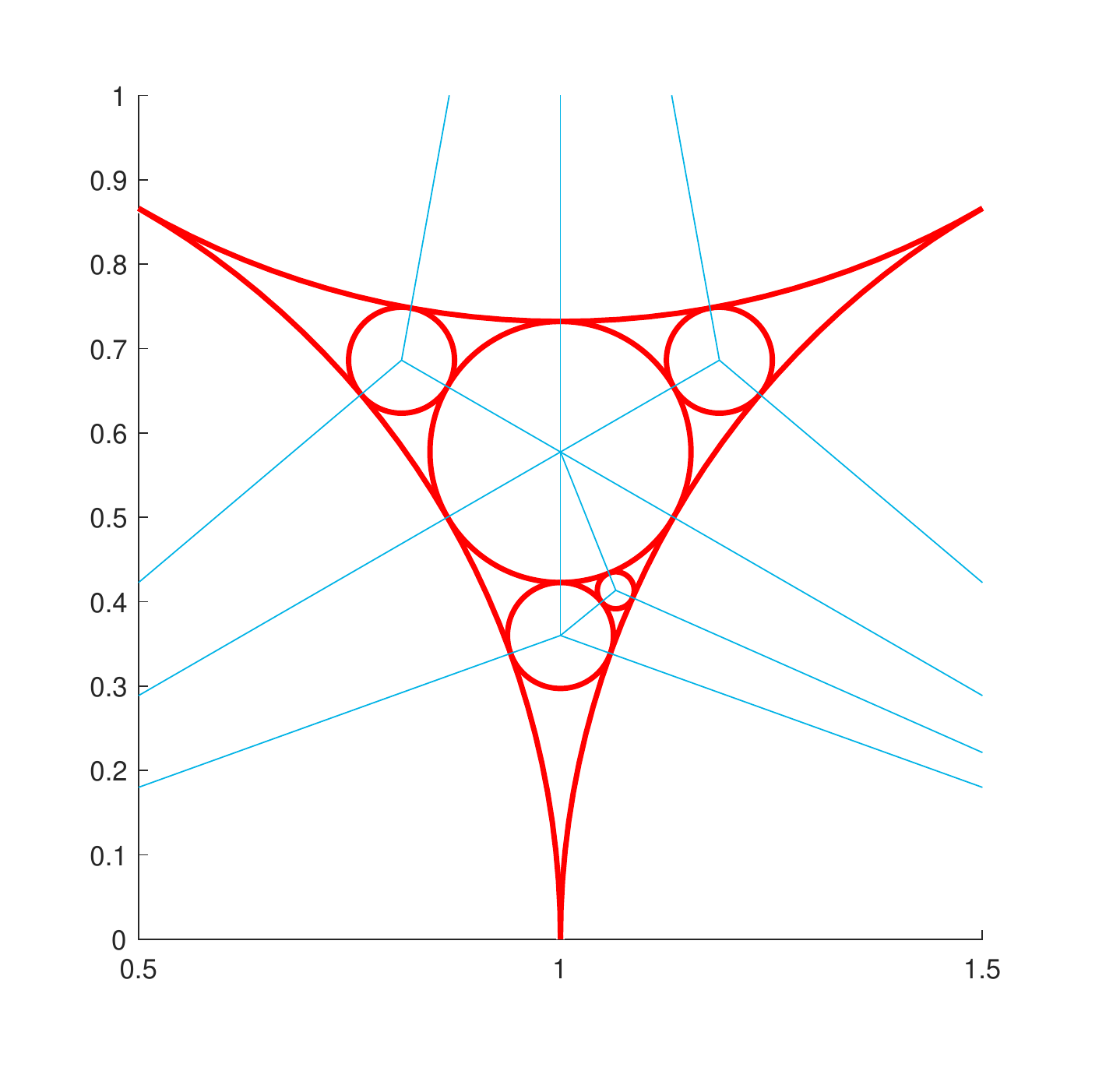}}
&
  {      \includegraphics[width=0.23\textwidth]{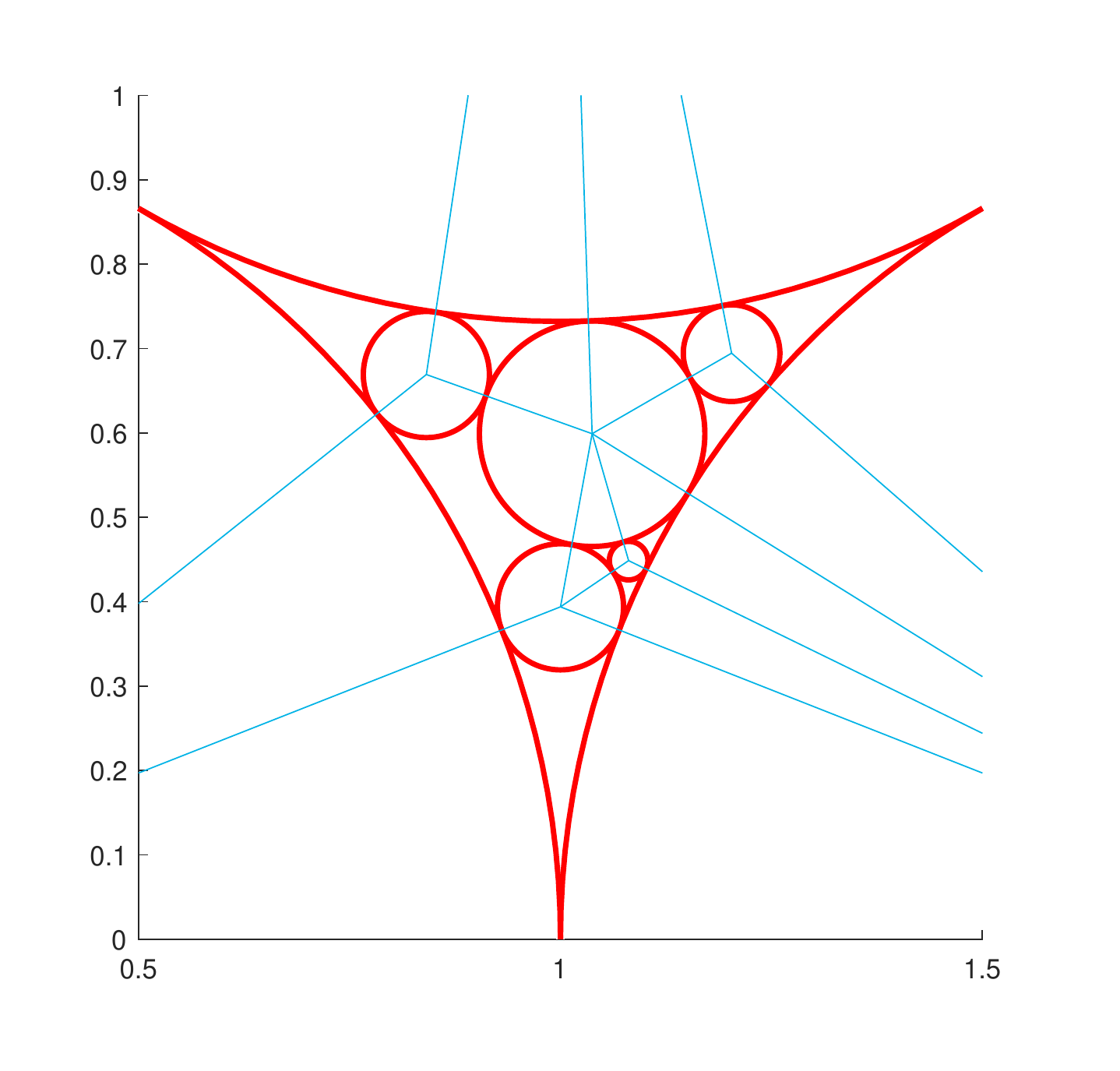}}
&
  {      \includegraphics[width=0.23\textwidth]{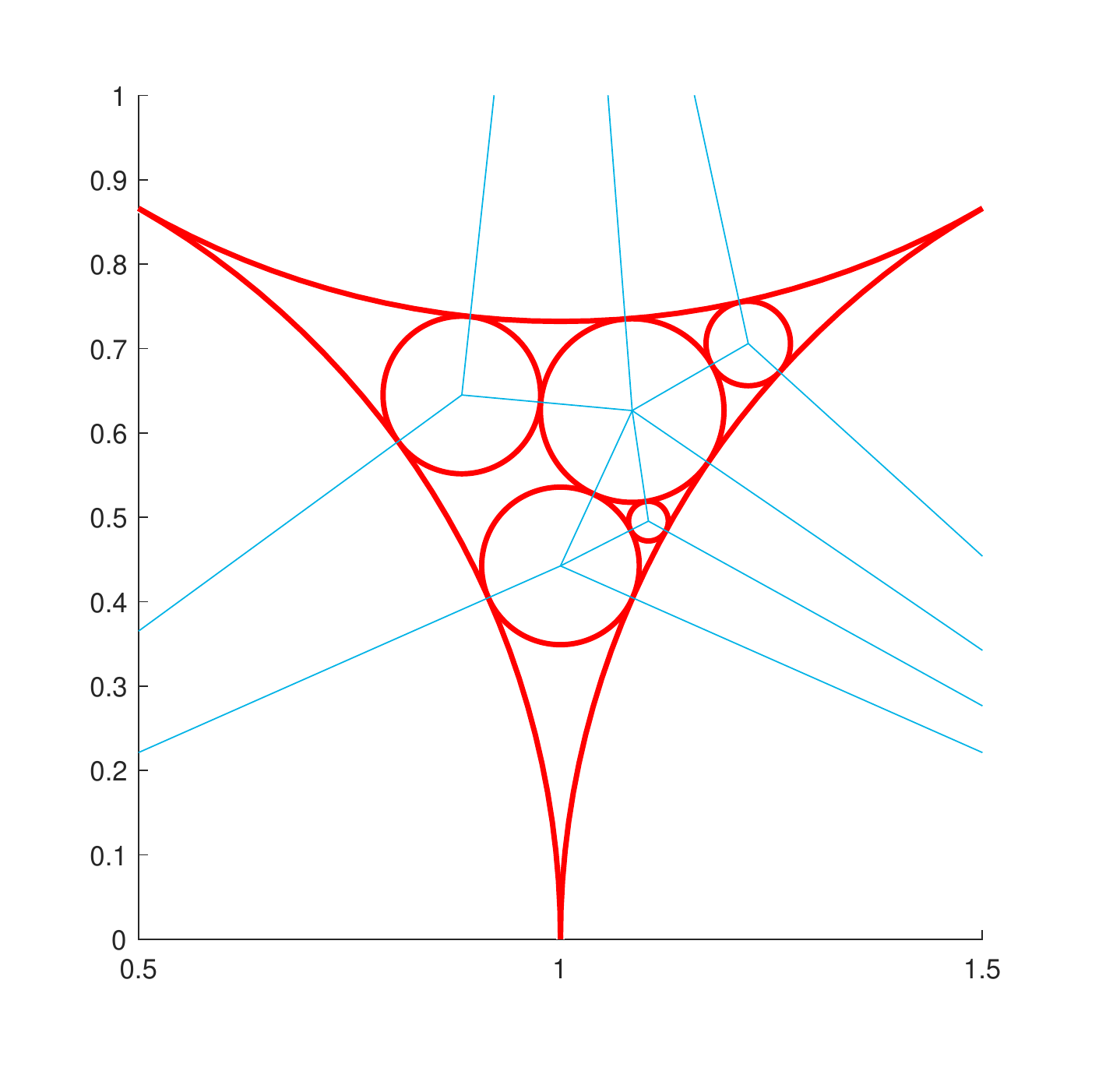}}
&
  {      \includegraphics[width=0.23\textwidth]{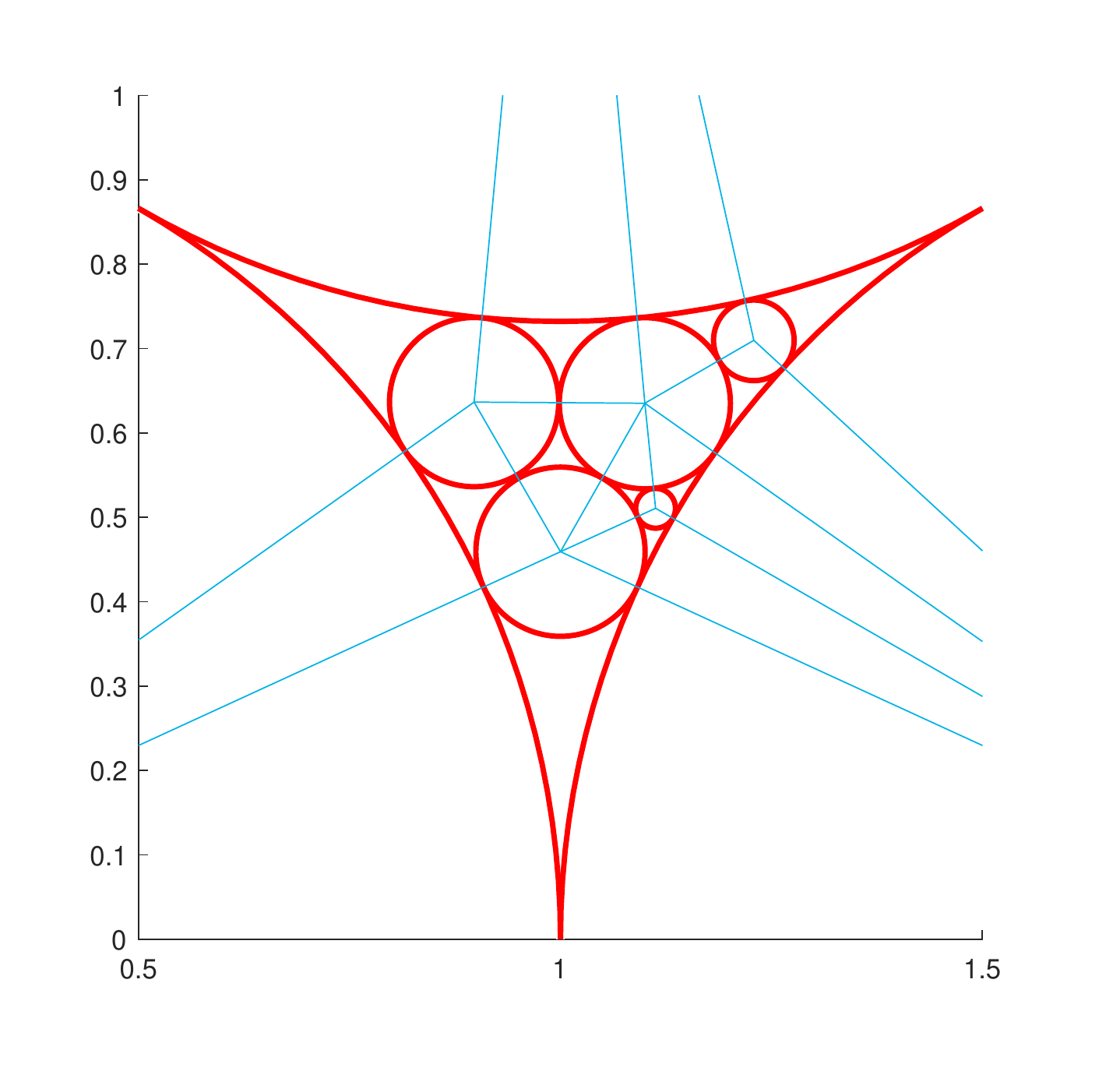}}
\\
  (a) & (b) &(c)&(d) \\
\end{tabular} 
    \caption{(a): A trilaterated  packing.  We are about to flip the edge $e^-$ connecting the central disk
    to the boundary disk on the lower left. (b) and (c): tridisk-contained almost
    triangulated packings (one quadrilateral) during the flow. (d): At the end of the flow,
    the cross-edge $e^+$  has made contact giving us a triangulated packing with a new
    graph.
    }
    \label{fig:flow}
\end{figure}

Together these ingredients can be used to establish the KAT circle packing theorem.
Our construction requires the inversion of a certain Jacobian matrix at each time step
of the flow, so we do not necessarily expect it to be faster than other practical 
circle packing techniques such as~\cite{ken}.
Interestingly, in our approach we always work with disks in the plane and never 
introduce or need to deal with any cone singularities in either the theory or the construction. 
Thus it provides a novel view of the circle packing problem.

\section{Preliminaries}

\begin{definition}
A (planar) \defn{disk configuration} $(\p,\r)$ 
is a \defn{sequence} (ie. an indexed family) 
of $n$ disks 
placed 
in $\RR^2$,
with centers at $\p := (\p_1, \ldots, \p_n)$ 
and with positive radii $\r := (r_1, \ldots, r_n)$. 
A (planar) \defn{packing} $(\p,\r)$ is a 
disk configuration
such that their interiors 
are mutually disjoint.  The \defn{contact graph} of a packing is the 
graph (over the same index set) that has one vertex
corresponding to each disk and one
edge corresponding to
each  pair of mutually tangent disks. 
\end{definition}

\begin{definition}
  Let $G$ be any graph with $n$ vertices, and let
  a  \defn{point configuration}
  $\p := (\p_1, \ldots, \p_n)$ 
  be a sequence of $n$ points in the plane.
  Define $(G,\p)$ to be the straight line drawing of the graph
  $G$ with vertex positions determined by  $\p$.
\end{definition}

\begin{lemma}
\label{lem:packemb}
  Let $G$ be the contact graph of some packing $(\p,\r)$. Then $G$ must be planar and $(G,\p)$ is a planar embedding.
\end{lemma}
\begin{proof}
This follows from the following two
facts:
Each of the line segments in $(G,\p)$ 
must be contained within its associated 
two touching disks.
None of the disks have overlapping interiors.
\end{proof}

\begin{definition}
A \defn{M{\"o}bius transformation} of the extended complex plane is a function
of the form
\ba
\phi(z) := \frac{az+b}{cz+d}
\ea
where $ad-bc\neq 0$.

\end{definition}

M{\"o}bius transformations are injective, orientation preserving.

We can apply a M{\"o}bius transformation to $\RR^2$ by identifying
it with $\CC$ (the transformation  may map one point to infinity,
and infinity to another point).

The \defn{generalized M{\"o}bius transformations} are generated by the
M{\"o}bius transformations and  reflections through  lines.

An \defn{inversion about a circle}, is a transformation where
 a point is sent to the point ``one over its distance'' from the center of the circle, 
 times the square of the radius of the circle on the ray through the center of the circle and the point.  
 It can be shown that a circle  inversion is a generalized M{\"o}bius transformation.


A generalized M{\"o}bius transformation $\phi$ maps an
input circle or a line to an output circle or a line.
If a circle $C_i$ maps to a circle $C_o$ with the same orientation under $\phi$,
then $\phi$ will map
the interior of $C_i$ (ie. a disk) to the interior
of $C_o$ (ie. a disk).
If a circle $C_i$ maps to a circle $C_o$ with the opposite orientation under $\phi$,
then $\phi$ will map
the interior of $C_i$ (ie. a disk) to the exterior
of $C_o$ (ie. the complement of a closed disk).
From its injectivity, we see that
a M{\"o}bius transformation maps a  packing to a  packing
(but up to one of the disks may become
a half plane or the complement of a  disk).

The following theorem is standard:

\begin{theorem}
\label{thm:mob3}
There is a unique M{\"o}bius transformation mapping 
any sequence of three distinct
"input"
points to any other sequence of three distinct "output"
points. 
This map varies smoothly as we alter the desired output point
locations.
\end{theorem}

\begin{definition}
  The \defn{inversive distance} between two disks,
  $(\p_i,r_i)$ and $(\p_j,r_j)$ is defined as
  \ba
  \frac{||\p_i-\p_j||^2 -r_i^2 -r_j^2}{2r_ir_j}.
\ea  
It evaluates to $1$ when the disks have 
external tangential contact.
\end{definition}

Inversive distances play well with M{\"o}bius transformations
as the following standard theorem states:
\begin{theorem}
\label{thm:invrMob}
The inversive distance between two disks is invariant for
generalized M{\"o}bius transformations.
\end{theorem}

\section{Tridisk-Contained  Packings}

In this section, we will use Theorem~\ref{thm:mob3} and
a standard construction (e.g,~\cite{marden})
to mod out 
generalized M{\"o}bius transformations.
The basic idea is to use a transformation to always put three chosen disks in to a  canonical form.
We go through the details somewhat carefully.

\begin{figure}[t]
    \centering
    \includegraphics[width=0.40\textwidth]{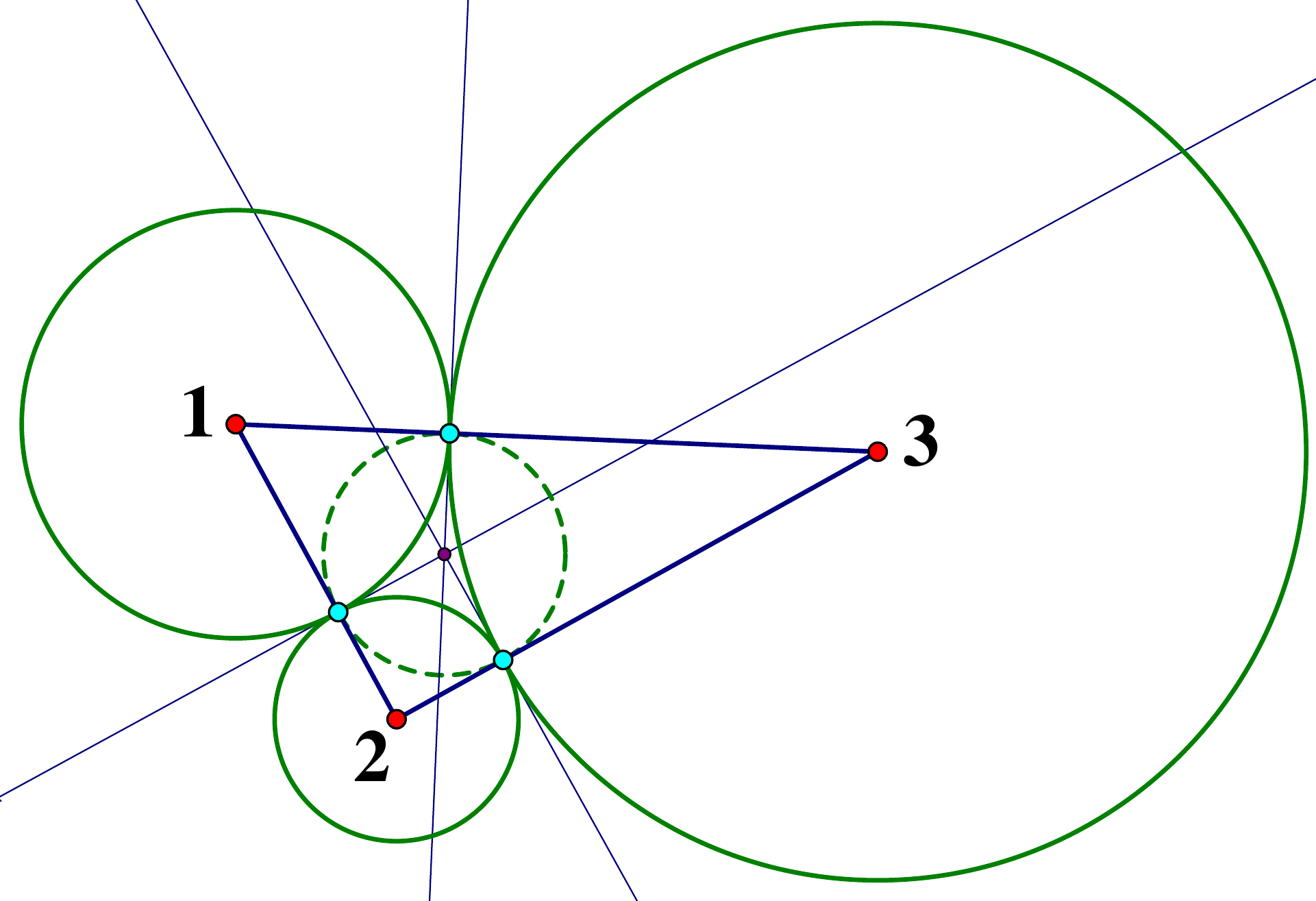}
    \caption{A set of three mutually 
    externally touching circles, with red center points,  have three contact points in blue.
    These contact points lie on a circle that is  called the
    \defn{incircle} (dashed) of the triangle determined by the centers of the circles. The blue contact points are
    not contained in any half of this incircle. 
    Going the other way, given any three points that
    are not contained in any half of their incident circle, we can associate a unique set of three mutually externally touching
    circles. 
    }
    \label{fig:intri}
\end{figure}

\begin{definition}
A \defn{tridisk} is a sequence of three non-overlapping disks in mutual tangential contact.
(See Figure~\ref{fig:intri}.)
The \defn{canonical tridisk} is the tridisk made of three 
disks of unit radius, where the disk centers lie on three  points
of a fixed equilateral triangle. (And to make this unique, we  fix an indexing of these three disks.)
(See Figure~\ref{fig:canon}.)

The \defn{incircle} of a tridisk is the circle touching the three
contact points of the tridisk. (See Figure~\ref{fig:intri}.)
\end{definition}

\begin{lemma}
\label{lem:tri1}
Given an input tridisk, there is a unique
M{\"o}bius transformation mapping it to the canonical tridisk.
\end{lemma}
\begin{proof}
  The input tridisk has a sequence of three  contact points. 
  There is a unique
  M{\"o}bius transformation $\phi$ mapping these to the sequence of three contact points of the
  canonical tridisk (Theorem~\ref{thm:mob3}).  
  A sequence of three externally touching circles are uniquely determined by such a sequence of  points (as in Figure~\ref{fig:intri}).
  So $\phi$ must map
 the three input touching circles 
 to the canonical three circles (with external 
 tangential contact).
 Now from injectivity, $\phi$ must map the interiors of the
 three distinct input disks to three disjoint regions.
 This disjointess ensures that $\phi$
 must map the interiors of the three input circles
 to the interiors of the three canonical circles.
\end{proof}

Note that the above transformation may map the 
interior of the tridisk's 
incircle to either the interior or the exterior of the canonical tridisk's
incircle. This depends on the relative orientations of the three contact points.

\begin{corollary}
\label{cor:tri2}
Given an input tridisk.
Let us select either the interior or the exterior of the
tridisk's incircle.
Then there is a unique
\textbf{generalized} M{\"o}bius transformation mapping the tridisk
to the canonical tridisk and which maps
the selected side of the tridisk's 
incircle to the interior of the canonical tridisk's
incircle.
\end{corollary}
\begin{proof}
We start with the M{\"o}bius transformation of Lemma~\ref{lem:tri1}.
If it has mapped the wrong side of the
input incricle's to the output incricle's exterior
we follow this by inverting through the canonical tridisk incircle. This must preserve the contact points, and
thus must leave each of the three circles invariant.
\end{proof}

The following will be needed later.

\begin{corollary}
\label{cor:tri3}
Given an input tridisk, with centers 
$(\p_1,\p_2,\p_3)$ and radii $(r_1,r_2,r_3)$, and
given any fixed displacement vector $\v \in \RR^2$, 
then for sufficiently
small $t$, 
there is a unique M{\"o}bius transformation $f_t$ that maps this tridisk to an output tridisk 
that has centers at $(\p_1,\p_2,\p_3+t\v)$  (the radii can change).
Moreover this transformation is smooth in $t$.
\end{corollary}
\begin{proof}
  Given the desired three (non-colinear)
  centers $(\p_1,\p_2,\p_3+t\v)$, there is a 
  always a unique set of radii $(r'_1, r'_2, r'_3)$, such that this creates an output tridisk.
  This output tridisk has a sequence of three contact points, giving us the unique M{\"o}bius transformation
  of Lemma~\ref{lem:tri1}. 
  All of the above steps are smooth.
   \end{proof}

\begin{definition}
Given a packing $P$ with contact graph $G$
that includes some triangular face $T$, we can 
\defn{mark the packing} by assigning, in some order, the labels $1$, $2$ and $3$, 
to the three disks corresponding to the vertices of $T$. This specifies a tridisk in $P$.
We refer to the other disks as \defn{unmarked}.
\end{definition}

\begin{definition}
A marked packing $P$,  where the mark-specified tridisk forms
the canonical tridisk, and all of the unmarked
disks lie inside of its ``tricusp'' region is called 
\defn{canonical-tridisk-contained}. We will shorten this to simply
\defn{tridisk-contained}.
\end{definition}

The next two lemmas tells us that we 
can use tridisk-containment as 
an effective modding process.

\begin{lemma}
\label{lem:fix}
Let $P$ be a packing of $n\ge 4$ disks with 
a $3$-connected contact graph $G$ that includes
some triangular face $T$.
Let $P$ be marked, specifying a tridisk.
Then there is a unique generalized M{\"o}bius transformation that maps 
$P$ to a tridisk-contained packing.
\end{lemma}
\begin{proof}
From Lemma~\ref{lem:packemb}, $(G,\p)$ is a planar embedding.
Since the tridisk corresponds to a triangular face of $G$, then in a packing, either all of the
other disks lie inside this triangle, or they all lie outside of it.
  The lemma then follows using Corollary~\ref{cor:tri2}.
\end{proof}

\begin{lemma}
\label{lem:fix2}
Let $P$ be a tridisk-contained
marked packing and with contact graph $G$.
If $P$ is the only 
packing with contact graph $G$, 
\textbf{up to a generalized 
M{\"o}bius transformation},
then 
$P$ is the \textbf {only} tridisk-contained  packing \textbf{with this marking} and with 
contact graph $G$.
The converse holds as well.
\end{lemma}
\begin{proof}
The case with $n=3$ is clear, so we will now assume that $n\ge 4$.

Suppose that $P$ is only packing with contact graph $G$, 
up to a generalized 
M{\"o}bius transformation.
Then, from uniqueness of the transform in Lemma~\ref{lem:fix}, 
it is the only tridisk-contained packing with this contact graph and marking.

For the other direction,
suppose there was a packing $Q$ unrelated to $P$ by a generalized M{\"o}bius 
transformation,
but with the same contact graph, $G$.
Let us take the marking of $P$, and use this to mark $Q$.
Then from Lemma~\ref{lem:fix}, there is a 
generalized M{\"o}bius transformation taking $Q$ to some tridisk-contained marked packing $Q'$.
By the assumption of no-M{\"o}bius-relationship, $Q'$ cannot equal $P$, but both are tridisk-contained with the same 
contact graph and marking.
This contradicts the assumed marked uniqueness of $P$.
\end{proof}

\section{Trilaterated Packings}

Next we describe a simple family of packings where
we can easily argue existence and uniqueness.

\begin{figure}[t]
    \centering
    \includegraphics[width=0.35\textwidth]{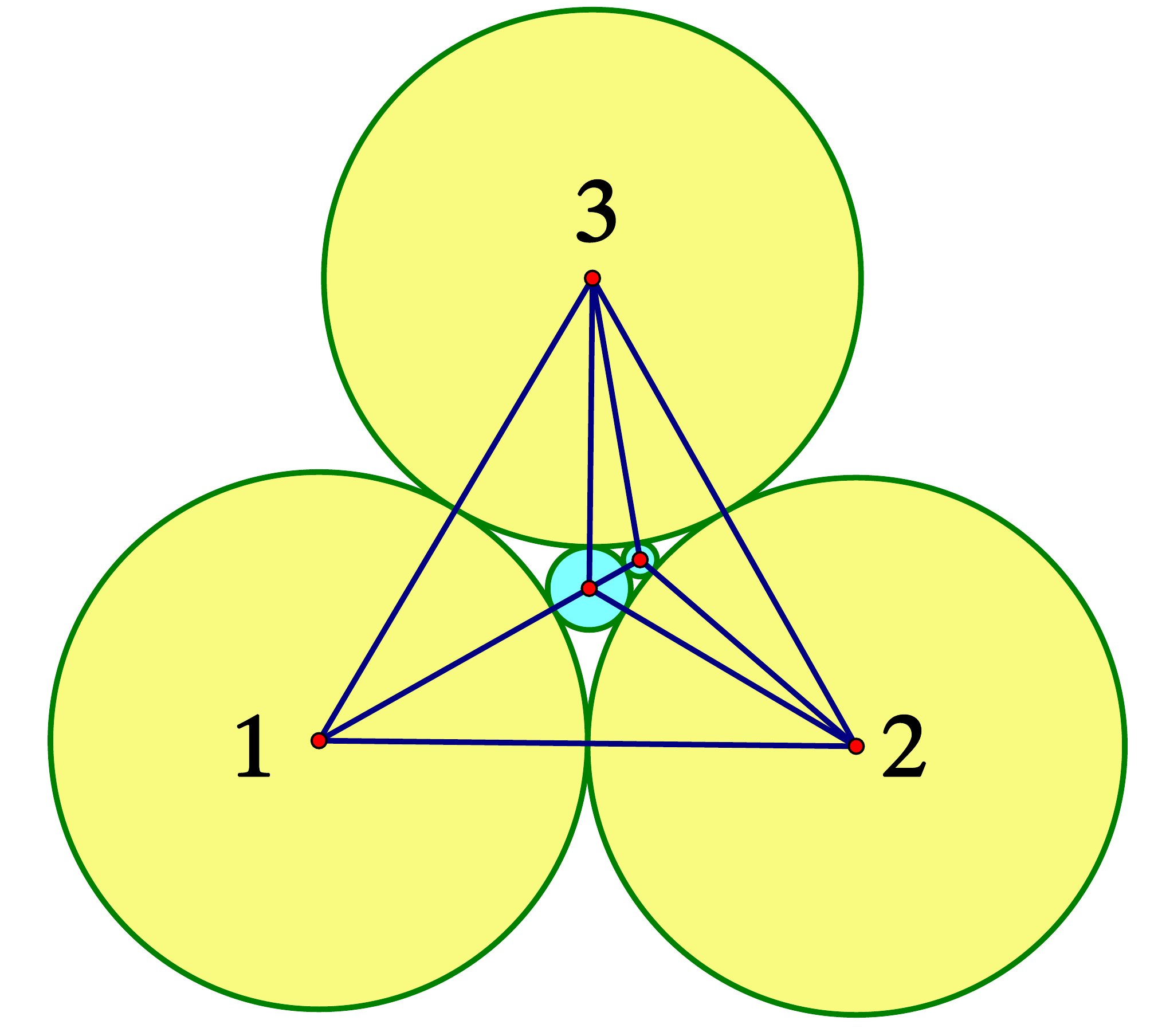}
    \caption{A trilaterated packing of $5$ disks. The outer
    three (yellow) disks form the canonical tridisk.}
    \label{fig:canon}
\end{figure}

\begin{definition}
To create a \defn{trilaterated}   packing 
with $n\ge 3$ vertices  we start with the canonical tridisk.
We then apply $n-3$ subdivision steps. Each subdivision
step adds one packing disk,  externally tangent to exactly
three already-existing disks.
We can create such a packing for any $n\ge 3$.
(See Figure~\ref{fig:canon}.) 
Such a packing is naturally marked by the canonical tridisk, and is tridisk-contained.
\end{definition}

\begin{lemma}
\label{lem:canunq}
Let $P$ be a trilaterated packing, and let $C$ be its contact graph.
Then $P$  is the unique packing with contact graph $C$, up to a
generalized
M{\"o}bius transformation.
\end{lemma}
\begin{proof}
Using Lemma~\ref{lem:fix2}, all we need to show is that  
$P$ is the
only tridisk-contained
packing with this marking and with contact graph $C$.

The proof then follows by induction starting with the unique canonical tridisk.
In particular, 
assume that $n \ge 4$ and that, 
by induction, the lemma holds for all trilaterated packings
with $n-1$ disks. A trilaterated packing on $n$ disks is obtained 
from a trilaterated packing 
on $n-1$ disks by adding in one final disk.
To result in a packing, the final disk must be placed in the
tricusp of its three neighbors (which are fixed by induction).
(If the new disk center were placed outside of the tricusp of its three contacting circles, 
but inside of the canonical tricusp,
then this would create a non-embedding for 
the straight-line drawing,
$(G,\p)$, violating Lemma~\ref{lem:packemb}.)
There is only one  center and radius for this last disk that 
creates the three required contacts, and is
inside the tricusp of its three neighbors.
\end{proof}

\section{Edge Flips}

\begin{definition}
A \defn{maximal planar graph}
is a planar graph such that adding any extra edge on the same vertex set
results in a non-planar graph.
\end{definition}

A  $3$-connected planar graph with $n\ge 3$ vertices has a well defined
set of faces.  
A maximal planar graph 
is always
$3$-connected, and each of its faces is a triangle.

\begin{figure}[t]
    \centering
   \includegraphics[width=0.55\textwidth]{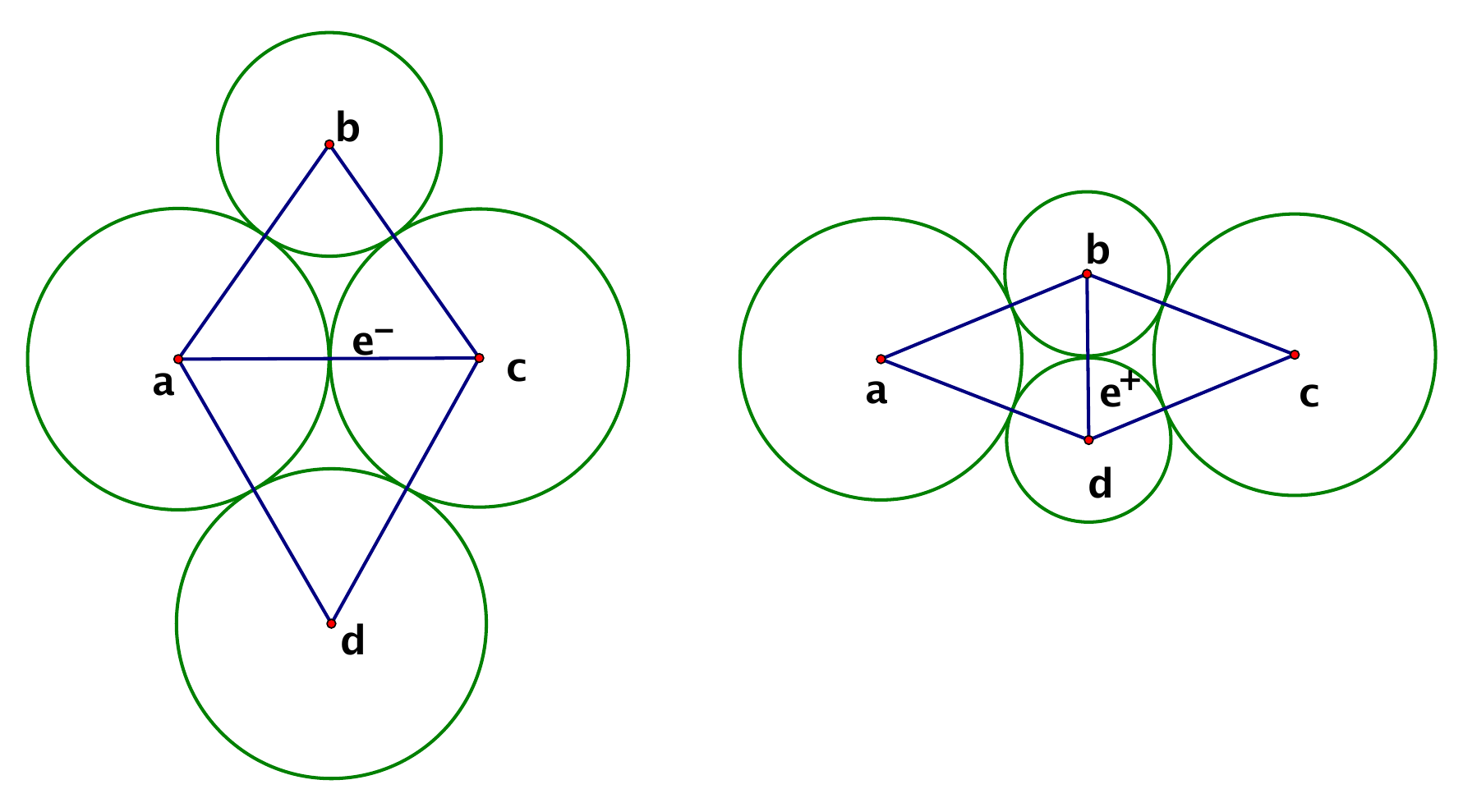}
    \caption{An edge flip.}
    \label{fig:flip}
\end{figure}

\begin{definition}
Let $G$ be a maximal planar graph with $n \ge 5$ vertices.
Let $e^-$ be an edge  of $G$ connecting two vertices denoted by $a$ and $c$.
This edge bounds two triangles in $G$. Denote by $b$ and $d$ the other two
 vertices of these two triangles. 

We say that $e^-$ is \defn{flippable} if there is no edge in $G$ connecting
$b$ and $d$. (See Figure~\ref{fig:flip}.)
\end{definition}

\begin{lemma}\label{lem:3c}
Let $G$ be a maximal planar graph and $e^-$ a flippable edge of $G$.
Then $G^-:=G\setminus e^-$ is $3$-connected.
\end{lemma}
\begin{proof}
The graph $G^-$ is a triangulation of the quadrilateral
on vertices $a,b,c,d$.  Such a triangulation must
be $3$-connected unless there is an edge 
connecting two non-consecutive vertices of the
quadrilateral
(e.g. ~\cite{laumond}). But we have removed the edge connecting
$a$ and $c$, and flippability presumes no edge between
$b$ and $d$.
\end{proof}

\begin{definition}
Let $G$ be a maximal planar graph and let $e^-$ be a flippable edge.
Then we call the graph $G^-:= G\setminus e^-$ an 
\defn{almost-maximal} planar
graph.  All of its faces are triangles, except for a single quadrilateral, with
vertices $a$, $b$, $c$ and $d$.
\end{definition}

\begin{definition}
Let $G^-$ be an almost maximal planar graph as above. Let $H$ be the graph
obtained from $G^-$ by adding the edge $e^+$ connecting $b$ and $d$. 
Then we say that $H$ was obtained from $G$ by an \defn{edge flip}
on $e^-$. 
\end{definition}

If $H$ is obtained from $G$ by an edge flip, then $H$ is a maximal
planar graph.

The following theorem~\cite{wagner,sleator} 
lets us use edge flips to navigate between
two graphs (and this can be done
while preserving the vertex indices).
\begin{theorem}
\label{thm:flip}
We can transform any maximal planar graph on $n$ vertices to any other maximal planar graph 
on $n$ vertices
through a finite sequence of edge flips.
\end{theorem}

\begin{definition}
  Let $G$ be a maximal planar graph. 
  We call any packing $(\p,\r)$ with contact graph $G$,
    a
    \defn{triangulated packing}.
  Let $G^-$ be an almost-maximal planar graph. 
  We call any packing $(\p,\r)$ with contact graph $G^-$,
    an
    \defn{almost-triangulated packing}.

        Given an almost-maximal packing with contact graph $G^-$,
    we can mark the three disks corresponding to  
    any triangular face of $G^-$ and transform the packing to obtain 
    a tridisk-contained, almost-maximal marked packing.
\end{definition}

\section{Bounding Radii}

  \begin{figure}[h]
    \centering
    \includegraphics[width=0.4\textwidth]{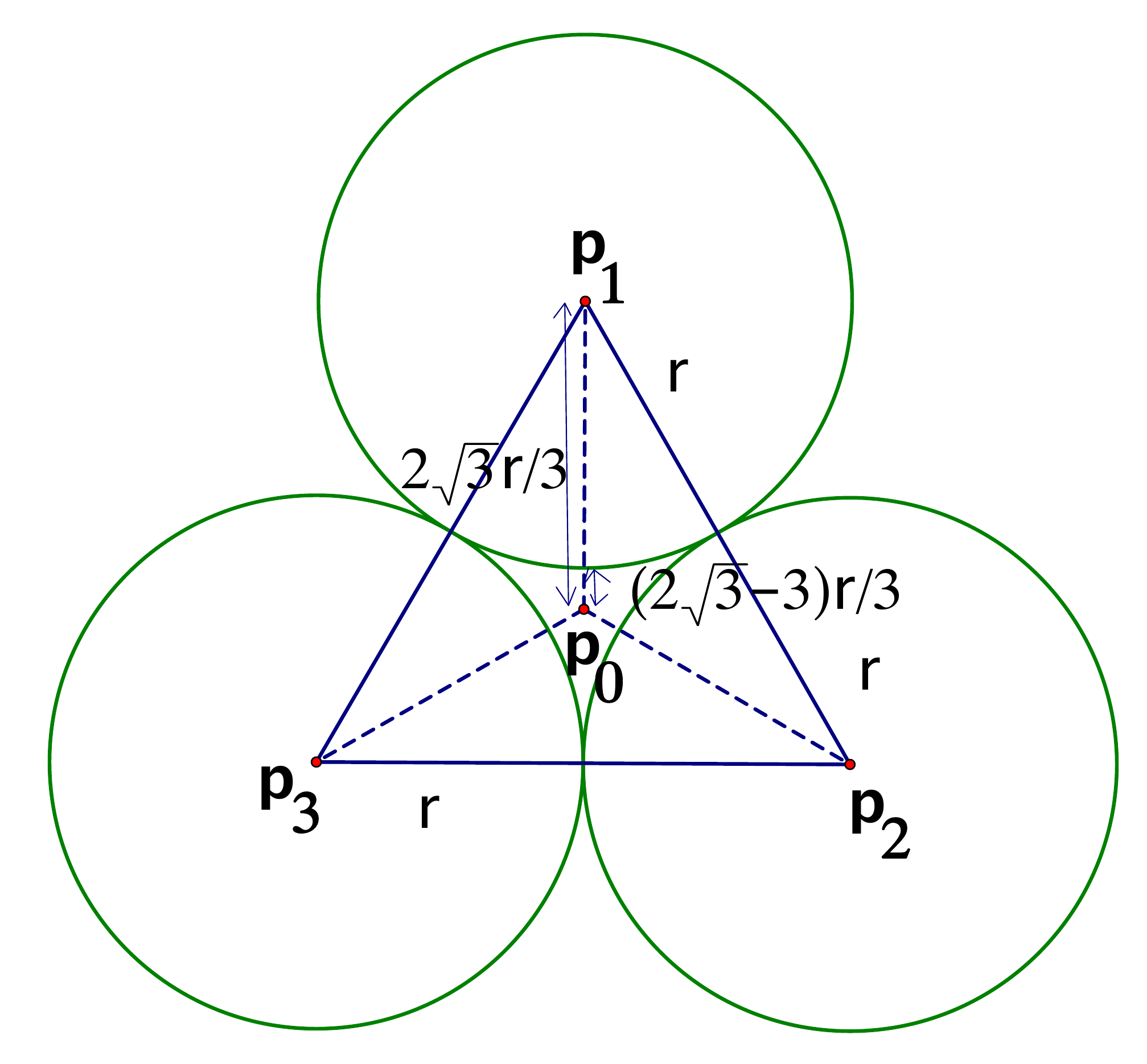}
    \caption{ All the inequalities of Lemma \ref{lemma:tensegrity}
    are equalities here.  Note that the distance from $\p_0$ to the circumference of each of the disks of radius $r$ is 
    $(2/\sqrt{3}-1)r$.
    }
    \label{fig:tensegrity}
\end{figure}

Next we will argue that in a tridisk-contained, almost-triangulated packing, 
all of our radii are bounded away
from $0$. 
  We start with the following:

    \begin{lemma}\label{lemma:tensegrity} There is no configuration of four points, $\p_0, \p_1, \p_2, \p_3$,  in the Euclidean plane and 
    real number $r>0$ with all the following properties: 
    \begin{itemize}
\item    For $i$ and $j$ in $\{1,2,3\}$ with $i\neq j$, 
    we have $|\p_i-\p_j|\ge 2r$. 
\item     For $i \in \{1,2,3\}$, we have  $|\p_i-\p_0| \le (2/\sqrt{3})r$.   
\item At least one of the above $6$ inequalities is strict.
\end{itemize}
  \end{lemma}
 
  The proof is by constructing a positive semi-definite (stress-energy) function whose global minimum occurs only when the above
  inequalities hold as equalities. 
  The energy is constructed so that for any configuration where the inequalities are  satisfied, the
  energy value is at least as small as this minimum value.
  See \cite{tensegrity, energy} for a discussion for this easy case and many other similar ones.   See Figure~\ref{fig:tensegrity}.

\begin{lemma}
\label{lem:tensegrity2}
Let there be $3$ packing disks, and a point $\p_0$, whose distance to the boundary of each of the disks is less than 
some $\epsilon$. Then at least one of the disks has radius $\le \epsilon/(2/\sqrt{3}-1)= (6.46..)\epsilon < 7\epsilon$.
\end{lemma}
 \begin{proof}
Suppose that all of the disks had radii $> \epsilon/(2/\sqrt{3}-1)$.
Then we could
shrink the radii of the two largest disks to the radius $r$ of the smallest disk, as in Figure \ref{fig:3-disks}
while maintaining a packing. Let us denote these new  disk centers as $\p_1$, $\p_2$, and $\p_3$.
For  $i$ and $j$ in $\{1,2,3\}$ with $i\neq j$, 
    we would have $|\p_i-\p_j|\ge 2r$. 
Meanwhile
  we  would have, for each $i \in \{1,2,3\}$,
$|\p_0-\p_i| < \epsilon + r \le (2/\sqrt{3}-1)r+r=(2/\sqrt{3})r$. 
This would contradict Lemma~\ref{lemma:tensegrity}.
\end{proof}

\begin{figure}[h]
    \centering
    \includegraphics[width=0.45\textwidth]{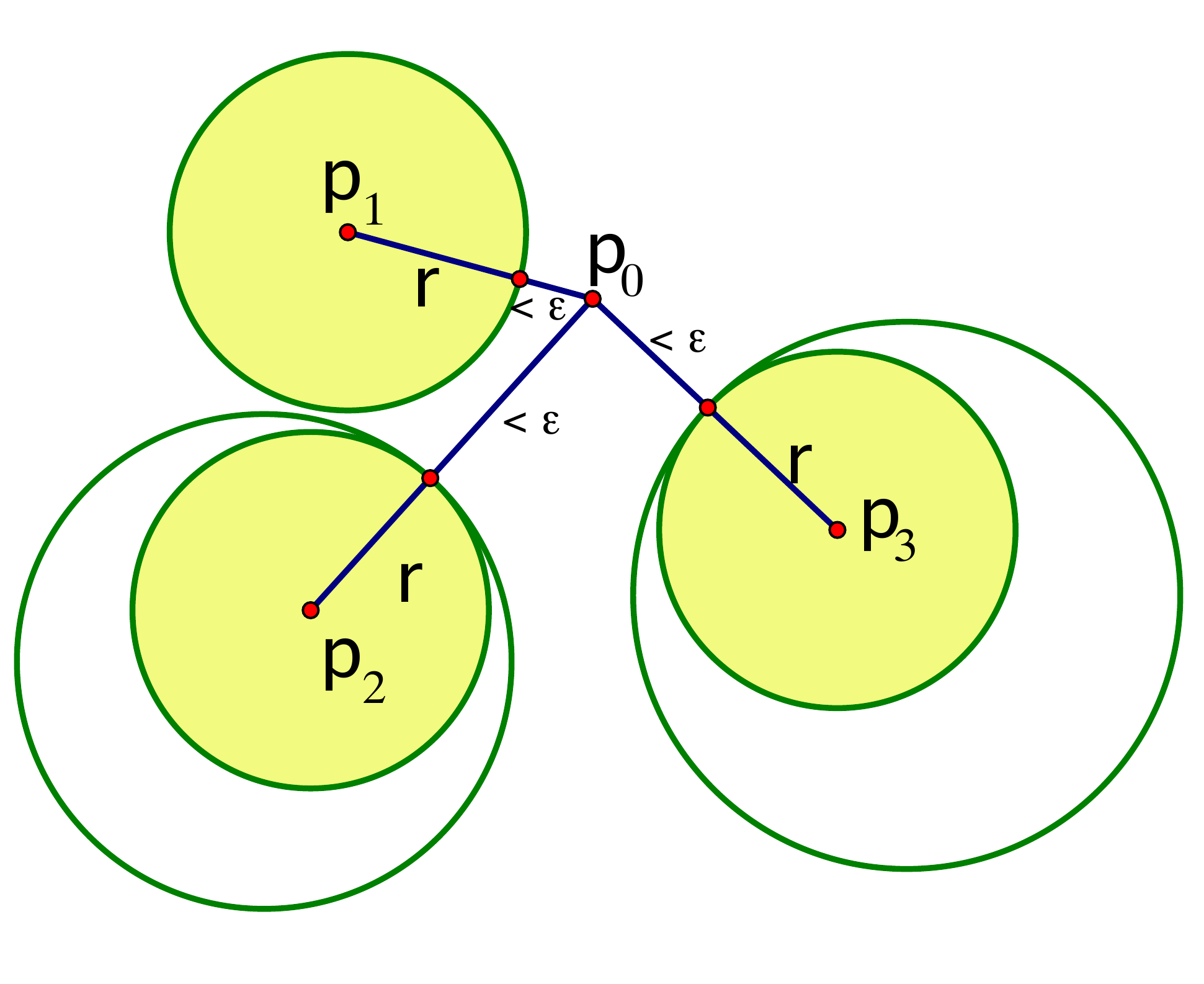}
    \caption{Here the point $\p_0$ is within $\epsilon$ of each of three disks all of radius at least $r$, the radius of the smallest disk.  The other disks have their radii shrunk to $r$, while keeping the nearest points to $\p_0$ the same.  The $\p_i$, are the centers of the shrunken disks. }
    \label{fig:3-disks}
  \end{figure}

And now we can prove our proposition:

\begin{proposition}
\label{prop:nozero}
Let $G$ be a fixed $3$-connected graph on $n$ vertices, (such as an almost-maximal planar graph).
Let $(\p,\r)$ be a packing with 
contact graph $G$. Suppose that at least $3$ of the disks have radii $\ge 1$.
 Then, all of the radii are bounded away
  from $0$ (where the bound depends only on $n$).
\end{proposition}
\begin{proof}

 Define $M_\epsilon$ to be a maximal connected set of vertices where
  all radii are $<\epsilon$ for any chosen $\epsilon$.
  Claim: If $\epsilon < 1$, $M_\epsilon$ is
  non-empty, and $M_{7n\epsilon}$ is a component containing $M_{\epsilon}$,
  then
$|M_{7n\epsilon}|\ge 1+|M_\epsilon|$.
  Proof: Since the disks in $M_\epsilon$
  are connected,
  they fit inside of some circle $C$ of radius $<n \epsilon$. 
From the assumed  $3$ large disks, we know that  
  $|M_\epsilon| \leq n-3$ (ie. this set  is a strict subset of $(\p,\r)$).
  Now we look at
  the neighbors of
  our set $M_\epsilon$ in $G$.
  From 3-connectivity, the set must have at least $3$ neighbors. 
  From Lemma~\ref{lem:tensegrity2}, at least one of these
  disks has radius less than $7n\epsilon$, establishing
  our claim.

  Now, let $\epsilon^* := \frac{1}{(7n)^n}$. If there is any
  disk with radius
  less than $\epsilon^*$, then by iteratively
  applying the previous paragraph, we find that all of disks
  have radii $<1$, contradicting our assumptions.

Finally, an almost maximal planar graph is $3$-connected from Lemma~\ref{lem:3c}.
\end{proof}

\begin{remark}
If we just know that there is at least $1$ large disk,
and we know that it is in the interior of a 
$3$-connected packing,
then we can use the same argument to prove that no other disk can be
arbitrarily small.
(For this we need to  argue that no set of $n-1$
disks, with $n-2$ of them sufficiently tiny can wrap around the known large disk.)
This then gives us a 
generalization of Rodin and Sullivan's "ring-lemma"~\cite{rodin}.
(We only assume $3$-connectivity, and there may not even be a single
``ring'' in the packing.)
\end{remark}

\begin{lemma}
\label{lem:G3}
Let $(\p,\r)$ be a tridisk-contained,  almost-triangulated packing with contact graph $G^-$.
Let  $ij$ be a vertex pair that is not an edge
in $G^-$, and is not $e^-$ or $e^+$. Then disks $i$ and $j$ must remain a bounded distance from each other by a constant depending only on the number
of vertices.
\end{lemma}
\begin{proof}
The graph  $G':=G^-\cup e_{ij}$
  is an  almost-maximal planar graph plus  one extra
  edge. But since $i$ and $j$ are not on some common face of $G^-$,
  the edge between does not subdivide any face of $G^-$ and so 
  $G'$ cannot be planar.
  (Fixing a (topologically unique) embedding of the $3$-connected graph 
  $G'$ on the sphere, 
  any curve starting at $i$ and ending at $j$ must
  enter and then exit one of the faces incident
  to $i$. The crossing at the exit point certifies
  non-planarity.)

  From Lemma~\ref{lem:packemb}, $(G^-,\p)$ is a planar embedding.
  Let us add in the segment connecting $\p_i$ and $\p_j$. From non-planarity, this must cross some (contact) edge $kl$ 
  of $(G^-,\p)$. The radii $r_i,r_j,r_k,r_l$ are 
  all bounded away from zero from 
  Proposition~\ref{prop:nozero}. This keeps disks $k$ and $l$ bounded away from contact.
\end{proof} 

\section{Infinitesimal Rigidity}
\label{sec:jacob}

In this section we prove an infinitesimal rigidity result about
inversive distances on
almost-triangulated packings. This,
along with Proposition~\ref{prop:nonsing},
will be  will be used  
to prove an existence and global rigidity result.

\begin{definition}
 Let $G$ be a maximal planar graph on $n$ vertices.
  Let $f$ be the mapping from $\RR^{3n}$
  to $\RR^{3n-6}$, that measures the \defn{inversive distance} along the
  edges of $G$ of a  disk configuration,
  (We do not require that these
  disks have disjoint
  interiors, but we do require the radii to be positive.)
    Explicitly we have:
\bna
\label{eq:inv}
  f(\p,\r)_{ij}:= 
  \frac{||\p_i-\p_j||^2 -r_i^2 -r_j^2}{2r_ir_j}.
\ena 
\end{definition}

  Let $J$ (at $(\p,\r)$) be the Jacobian of $f$.
  Its row corresponding to an edge $ij$ of $G$
has the following pattern~\cite{bbp}

       \bordermatrix{
        & \cdots & \p_i& \cdots & \p_j & \cdots & r_i & \cdots & r_j \cr 
        & \cdots\, 0\, \cdots & 
        \frac{\p_i - \p_j}{r_ir_j} &  
        \cdots\, 0\,\cdots & 
        \frac{\p_j - \p_i}{r_ir_j} &  
        \cdots\, 0\,\cdots & 
    \frac{r_j^2-r_i^2 - ||\p_i-\p_j||^2}{2r_i^2r_j}
        &
         \cdots\, 0\,\cdots & 
    \frac{r_i^2-r_j^2 - ||\p_i-\p_j||^2}{2r_j^2r_i}
             } 

\noindent where the first row above labels the matrix columns.

\bigskip

When disks $i$ and $j$ have no interior overlap, 
(indeed, when $\p_i$ is not inside of disk $j$)
we can verify that 
\bna  
\label{eq:pos}
    \frac{r_j^2-r_i^2 - ||\p_i-\p_j||^2}{2r_i^2r_j} \le 0.
\ena

When disks $i$ and $j$ are 
in tangential
contact, their inversive distance is $1$
and we have $||\p_i-\p_j||=r_i+r_j$ giving us
\bna  
\label{eq:contact}
    \frac{r_j^2-r_i^2 - ||\p_i-\p_j||^2}{2r_i^2r_j} =
    \frac{-r_i-r_j}{r_ir_j}  = 
    \frac{- ||\p_i-\p_j||}{r_ir_j}.  
\ena  


In this section, 
we will establish the following:
\begin{proposition}
  \label{prop:nonsing}
  Let $G$ and $H$ be maximal planar graphs
  related by flipping an edge $e^-$.
  Let $(\p,\r)$ be a (triangulated or almost-triangulated)
  packing with contact graph of either
  $G$, $G^-$, or $H$.
  Then the Jacobian, $J$, (defined using the edges of $G$)
  has rank $3n-6$.
\end{proposition}

Our argument in this section is a variation on ideas explored in~\cite{sticky,bbp}. 
The idea is to show that $J$ 
cannot have any non-zero cokernel
vector.
The main extra difference here is that
when the contact graph is $G^-$ or $H$, then the edges of $G$
are not in contact. We will argue that since $G^-$ and $H$ are
similar enough to $G$, this will not destroy the argument.

\begin{definition}
Given a 3-connected  planar graph $G$ with $n$ vertices and $m$ edges,
and given  a 
vector $\v$ in $\RR^{m}$, we assign 
a sign  $\{+,0,-\}$ to each undirected edge $ij$ using 
the sign of $\v_{ij}$. This assignment gives us an
associated \defn{sign vector} $\s$.
\end{definition}

\begin{definition}
Given a sign vector $\s$, 
we define the \defn{index} $I_i$ at vertex $i$, as 
the number of times the sign changes as we traverse the 
edges in order around vertex $i$, ignoring zeros.
(We use the unique face structure of a 3-connected planar graph to 
define the cyclic ordering of edges at each vertex.)
\end{definition}

The index $I_i$ is always even.

The following is a variation on
Cauchy's index lemma, which can be proven using Euler's 
formula. For a proof, see e.g.,~\cite[Lemma 5.2]{gluck}
or~\cite[Page 87]{AlexConv}.
\begin{lemma}
\label{lem:topo}
Let $G$ be a planar 3-connected graph  with $n$ vertices. Let
$\s$ be a  sign vector in $\RR^{m}$. Let $n'$ be the number of vertices that
have at least one non-zero signed edge.
Then $\sum_i I_i \leq 4n'-8$. 
\end{lemma}

\begin{lemma}
\label{lem:geom0}
Let $G$ be a maximal planar graph on $n$ vertices.
Let $(\p,\r)$ be a configuration of $n$ disks. 
Let $\v$ be a non-zero cokernel vector of $J$, giving us an associated sign vector.
Let '$i$' be the index of a disk that  has no interior overlap with its neighbors
in $G$ and such that it has at least one non-zero signed edge in $\v$.
Then its associated index  $I_i$
is at least $2$.
\end{lemma}
\begin{proof}
If vertex $i$ has at least one non-zero sign and no sign changes, then it 
cannot satisfy 
Equations~(\ref{eq:pos}) while having $\v$
annihilate the column of $J$ corresponding to
$r_i$.
\end{proof}

\begin{definition}
Let $G$ be any $3$-connected planar graph on $n$ vertices and $\p$ a configuration of $n$   points.
  Let $(G,\p)$  be its straight line drawing
  (not necessarily an embedding). We say that
  a vertex has \defn{out of order edges}
  if the drawn (either clockwise or counterclockwise) cyclic order 
  of its adjacent edges
  does not match
  their combinatorial cyclic order in $G$. 
  \end{definition}

\begin{definition}
Let $G$ be a maximal planar graph on $n$ vertices.
Let $(\p,\r)$ be a configuration of $n$ disks.
We call a disk with index `$i$' \defn{$G$-like}
if the edges of vertex $i$ in $(G,\p)$
are in order  and if for all of the edges
$ij$ in $G$ incident to
vertex $i$, the disks $i$ and $j$
are in tangential contact in $(\p,\r)$.
\end{definition}

\begin{lemma}
\label{lem:2oo}
Let $G$ be a maximal planar graph.
  Let $(\p,\r)$ be a triangulated or 
   almost-triangulated packing with contact graph
  $G$, $G^-$ or $H$. 
  Then there are at most two  
  disks that are not $G$-like.

\end{lemma}
\begin{proof}
The contact graph $G'$ of
$(\p,\r)$ is one of $G$, $G^-$ or $H$,
So Lemma~\ref{lem:packemb}  
tells us that $(G',\p)$ is embedded.
$G^-$ is a subgraph of $G'$
so $(G^-,\p)$ is also embedded. It is $3$-connected 
by Lemma~\ref{lem:3c}.
Thus all of the edges of $(G^-,\p)$ are 
 drawn
in order and have disks with tangential contact.
The graph $G$ includes only one more edge,
which can affect only two disks.
\end{proof}

\begin{definition}
Given a graph $G$ with $n$ vertices
and $m$ edges, and a configuration $\p$ of $n$ points.
A vector
$\omega \in R^{m} $ 
satisfies the \defn{equilibrium condition} at vertex $i$ if
\begin{eqnarray}
\label{eq:coker1}
\sum_j \omega_{ij} (\p_i-\p_j)&=&0\\
\label{eq:coker2}
\sum_j \omega_{ij} \|\p_i-\p_j\|&=&0.
\end{eqnarray}
The sums in \eqref{eq:coker1}--\eqref{eq:coker2}
are over neighbors $j$ of $i$ in $G$.
We use the index ${ij}$ to index
one of the $m$ edges of $G$.
\end{definition}
(This notion was studied in~\cite{sticky,bbp}.
See also~\cite{lam2}.)

\begin{lemma}
\label{lem:cokeq}
Let $G$ be a maximal planar graph on $n$ vertices.
Let $(\p,\r)$ be a configuration of $n$ disks.
Let $\v$ be in the cokernel of $J$ (defined using the edges of $G$).
Let us define the \defn{
associated stress vector} $\omega$ as
$\omega_{ij} := r_ir_j \v_{ij}$.
Suppose disk $i$ has  tangential
contact with its neighbors in $G$.
Then $\omega$ satisfies the equilibrium condition at vertex $i$.
\end{lemma}
\begin{proof}
This follows 
from nature of the entries in the
associated columns of $J$ and
using 
Equation (\ref{eq:contact}).
\end{proof}

\begin{lemma}[\cite{sticky}]
\label{lem:geom}
Let $G$ be a maximal planar graph on $n$ vertices.
Let $(\p,\r)$ be a configuration of $n$ disks.
Let $\v$ be a cokernel vector of $J$ giving us an associated sign vector.
Let $i$ be a $G$-like disk.
 Then the index $I_i$ 
is at least $4$.
\end{lemma}
\begin{figure}[t]
    \centering
    \includegraphics[width=0.25\textwidth]{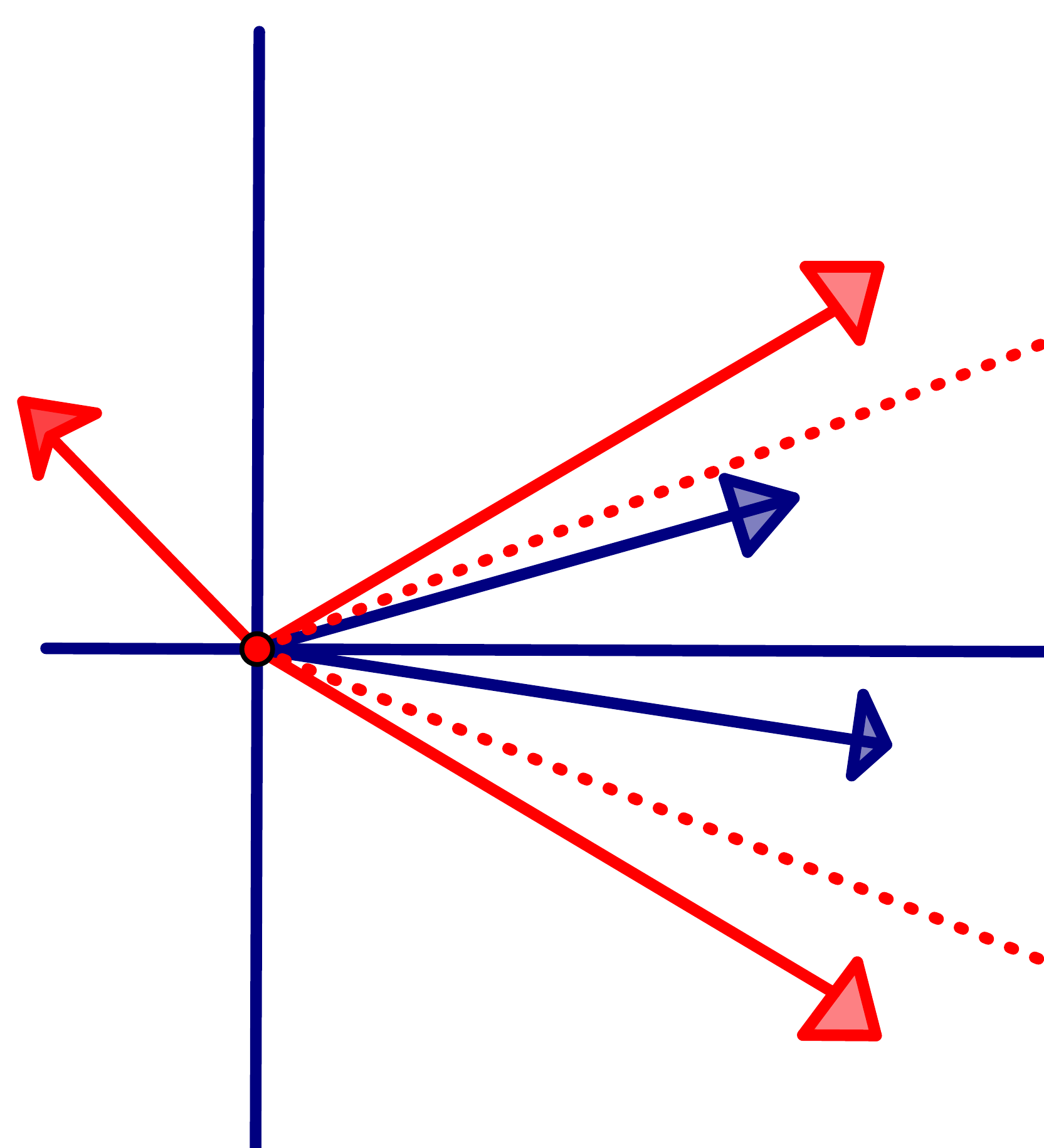}
    \caption{Illustration of the proof of Lemma \ref{lem:geom}.  This 
    vector configuration  has only two sign changes (blue are negative signs and 
    red are positive).}
    \label{fig: wedge}
\end{figure}
\begin{proof}
In what follows, we ignore any edge with a zero coefficient in $\v$.

From Lemma~\ref{lem:geom0}, $I_i\ge 2$.
So now we will suppose that $I_i=2$.
and arrive at a contradiction.

Using the assumed $G$-like property
of disk $i$, we can apply
Lemma~\ref{lem:cokeq}.
Thus the associated stress vector
 vector $\omega$,
where $\omega_{ij} := r_ir_j \v_{ij}$
 must satisfy the 
equilibrium condition at vertex $i$.

With the supposed $2$ sign changes, the edges 
with one of the signs (say $-$) must
be in a cone of angle $2\theta < \pi$ (for example, the dotted red
cone in Figure~\ref{fig: wedge}).
(Here we use the assumption that these edges are drawn in order in $(G,\p)$.)

Euclidean transforms have no effect on Equations 
(\ref{eq:coker1}) and  (\ref{eq:coker2}),
so without loss of generality , we may assume that the positive part of the  $x$-axis
is the bisector of the cone.  
The $2$D equilibrium condition of Equation (\ref{eq:coker1}) must hold 
after projection along any direction, including onto the $x$-axis,
since \eqref{eq:coker1} is invariant under any affine transformation
(see, e.g., \cite{MR2504741}).

Let $N^+$ denote the neighbors of $i$ connected by edges with 
positive sign and $N^-$ the neighbors connected by negatively signed
edges.
Let $p^x_i$ be the $x$-coordinate of the point $\p_i$.
We then get:
\begin{eqnarray*}
\sum_{j \in N^+} \omega_{ij} (p^x_i-p^x_j)
=
\sum_{j \in N^-} -\omega_{ij} (p^x_i-p^x_j).
\end{eqnarray*}
But for $j \in N^+$ (outside the cone), we have 
\begin{eqnarray*}
(p^x_i-p^x_j) < \cos(\theta)\; \|\p_i-\p_j\| 
\end{eqnarray*}
while for $j \in N^-$ (inside the cone),  we have 
\begin{eqnarray*} 
(p^x_i-p^x_j) > \cos(\theta)\; \|\p_i-\p_j\|.
\end{eqnarray*}
Putting these estimates together we have 
\[
    \sum_{j\in N^+} \omega_{ij} \cos(\theta)\; \|\p_i-\p_j\| 
    > \sum_{j \in N^+} \omega_{ij} (p^x_i-p^x_j)
    = \sum_{j \in N^-} -\omega_{ij} (p^x_i-p^x_j)
    > \sum_{j\in N^-} -\omega_{ij} \cos(\theta)\; \|\p_i-\p_j\| 
\]
which means that Equations \eqref{eq:coker1} and \eqref{eq:coker2} 
cannot 
hold simultaneously.
\end{proof}
See Figure~\ref{fig: wedge} for an illustration of this argument.

\vspace{.1in}
And now we can prove our Proposition.
\begin{proof}[Proof of Proposition~\ref{prop:nonsing}]

Suppose that $J$ has a non-zero co-kernel vector $\v$.
Let $n'$ be the number of disks
with at least one incident edge in $G$
with a non-zero $\v$ value.
From Lemma~\ref{lem:2oo}, at most two of these
$n'$ disks are 
non $G$-like, and at least $n'-2$
are $G$-like.
Summing over the indices of 
$n'-2$  disks 
that are $G$-like,
and using Lemma~\ref{lem:geom}, we obtain 
the value of at least $4n'-8$. 
The remaining two of the $n'$ disks 
do not overlap and so add at least 
$4$ more to the sum using Lemma~\ref{lem:geom0}, giving us
at least $4n'-4$.
But this contradicts Lemma~\ref{lem:topo}, 
which ensures
$\sum_i I_i \leq 4n'-8$. 
\end{proof}

\subsection{Square Jacobian}

The matrix $J$ tells us  how differential 
changes, $(\p',\r')$ lead to differential
changes, $\f'$, 
in the inversive distances along the edges of 
a maximal planar graph $G$.
The fact that, for an (almost-)triangulated packing,  $J$ has full row rank tells us that all differential
changes of inversive distances on the edges 
of $G$ are achievable. 

From this row rank and the matrix size, we see that 
the kernel of $J$ is $6$ dimensional, corresponding exactly
to the $6$ M{\"o}bius degrees of freedom. 

Next we wish to mod out the M{\"o}bius degrees
of freedom. We will do this by marking and pinning:

\begin{definition}
 Let $(\p,\r)$ be a packing with contact graph of either
  $G$, $G^-$, or $H$.
Pick a triangular face of $G^-$, and \defn{mark}
the corresponding three disks.
(Ultimately these three disks will form our canonical tridisk boundary.)
We will also refer to the three edges in 
$G$,
corresponding to this triangle as \defn{marked}.
We then \defn{pin} the centers (but not radii) of our three marked disks.
We call this a \defn{center-pinned}  packing.

The affect of center pinning on $J$ is simply to discard 
the corresponding
$6$ columns, giving us a square matrix of size $3n-6$ we call the \defn{center-pinned Jacobian} $J_c$.
\end{definition}

\begin{lemma}
\label{lem:nonsing2}
 Let $(\p,\r)$ be a  center pinned packing with contact graph of either
  $G$, $G^-$, or $H$.
Then its center-pinned Jacobian (defined by edges in $G$) is a non-singular matrix.
\end{lemma}
\begin{proof}
From Proposition~\ref{prop:nonsing}, 
$J$ has rank $3n-6$.

From  Corollary~\ref{cor:tri3} we can fix the $5$ coordinates corresponding to say, $\p_1$ and $\p_2$, and the
$x$-coordinate of $\p_3$, and smoothly move the $y$-coordinate of $\p_3$ using a a family of
M{\"o}bius transformations $\phi_t$. Differentiating $\phi_t$  by $t$ at $t=0$ gives us a vector 
$(\p',\r')$ which, from Theorem~\ref{thm:invrMob},
must be in the kernel of $J$. 
By construction $(\p',\r')$
has $0$ entries corresponding to the $5$ fixed coordinates. Thus the column of $J$ corresponding
the $y$-coordinate of $\p_3$
must be linearly dependent on the other $3n-6$ columns. So this column can be
removed without changing the rank of the Jacobian. The same reasoning applies symmetrically
to all $6$ of these columns. 
Removing these $6$ columns gives us the center-pinned Jacobian, $J_c$ which is square and has rank $3n-6$.
\end{proof}

Next, we note that $J_c$ has a useful block form.
\begin{definition}
The three rows of $J_c$ corresponding to the marked edges are only supported
on the three columns corresponding to the three radii corresponding to the marked disks (the positional 
columns for these disks have already been removed).
Let us now also pin these three radii, resulting in a \defn{pinned packing}
By removing these three rows and columns, we obtain a square  matrix of size $3n-9$. 
We call this the \defn{pinned Jacobian} $J_p$. Its rows correspond to the unmarked edges of $G$ and its columns correspond
to the positions and radii of the unpinned vertices.
\end{definition}

\begin{lemma}
\label{lem:nonsing3}
 Let $(\p,\r)$ be a pinned
  packing with contact graph of either
  $G$, $G^-$, or $H$.
Then its pinned  Jacobian (defined by the edges of $G$) is a non-singular matrix.
\end{lemma}
\begin{proof}
This follows from Lemma~\ref{lem:nonsing2}
and its block form just described.
\end{proof}

\section{Packing Manifold}

\begin{definition}
\label{def:sgm}
 Let $G$ and $H$ be maximal planar graphs
  related by flipping an edge $e^-$, and 
 $G^-$ their common,  almost-maximal planar graph.
Pick a triangular face $T$ of $G^-$, and \defn{mark}
(in $G$, $G^-$, and $H$)
the corresponding three vertices, in some order, with the labels $1$, $2$ and $3$.
We also refer to the three edges of $T$  as marked.
Given 
any packing $P$ with contact graph $G$, $G^-$ or $H$,
we mark three of its disks 
using these vertex markings.

Define the \defn{packing set} $S_{G^-}$, to be the subset of  configurations of $n$ disks
that are
packings and where
the contact graph is exactly the graph $G^-$ (no extra contacts), and 
that are 
tridisk-contained under the above marking. There are $3$ common pinned disks in $S_{G^-}$
 
Let us consider the \defn{configuration space} of the unpinned $n-3$ disks  as $\RR^{3n-9}$.

\end{definition}

This set satisfies a set of 
strict radius inequalities
$r_i>0$, 
strict inequalities over the non-edges, and 
equalities over the contact edges.
Given these constraints, the tridisk containment can then be enforced by the 
strict inequality
requirement that
the unmarked-disk centers are in the open interior of the tricusp region.

\begin{definition}
The \defn{frontier points} of $S_{G^-}$
are define as $\overline{S_{G^-}}\backslash S_{G^-}$.
\end{definition}

Now we establish some geometric properties
of $S_{G^-}$.

\begin{lemma}
\label{lem:mani}
$S_{G^-}$ is a $1$-dimensional smooth manifold.
\end{lemma}
\begin{proof}
Let $P$ be a point in $S_{G^-}$. 
In a sufficiently small neighborhood of $P$ in configuration space, all of the radii must remain positive,
all non-edge connected disk pairs must remain disjoint and all of the unmarked-disk centers must be inside of the tricusp. 
(These are open conditions.)
Thus locally, we need to only 
consider the 
$3n-10$
inversive distances equalities
of Equation (\ref{eq:inv}).

The partials of these constraints are described in the
$3n-10$ rows of $J_p$ corresponding 
to the unmarked edges of $G^-$.
By Lemma~\ref{lem:nonsing3}, these rows are linearly 
independent and thus, from the implicit function theorem, 
$S_{G^-}$, restricted to some neighborhood of $P$, is a smooth manifold of 
dimension $1$.
\end{proof}
See~\cite{het} for some related and more general results.

\begin{lemma}
\label{lem:bounded}
$S_{G^-}$ is bounded, and so its closure is compact.
\end{lemma}
\begin{proof}
The tridisk constraint
bounds the coordinates of $\p$ and $\r$.
\end{proof}

\begin{lemma}
\label{lem:frontcontact}
A frontier point of $S_{G^-}$ 
can only be a tridisk-contained maximal triangulated  packing with
contact graph $G$ or $H$ (either $ac$ or $bd$ in contact.)
\end{lemma}
\begin{proof}
Let us consider the possible frontier points of $S_{G^-}$.
From Proposition~\ref{prop:nozero}, for a packing in 
$S_{G^-}$, every radius must be bounded above  zero. 
From Lemma~\ref{lem:G3}, for a packing in 
$S_{G^-}$, every disk pair corresponding to any non-edge,
except for $ac$ or
$bd$, must be bounded away from contact.
Meanwhile, any disk configuration with positive radii, where  
some edge of $G^-$  is not in tangential contact, 
cannot be approached while satisfying the (closed) equality conditions of $S_{G^-}$.
Finally, the above conditions also keep all of the disk centers bounded away from the tricusp boundary.
The lemma follows by this process of elimination.
\end{proof}

Ultimately, we will see below as a result of
our
flow argument, that each connected component of 
$S_{G^-}$ has one frontier point that
is a tridisk-contained packing with  contact graph $G$
and another frontier point that is a tridisk-contained packing with the flipped contact graph $H$.
Moreover a packing $P$  with contact
graph $G$ (or resp. $H$) is in the closure
of exactly one connected component of $S_{G^-}$.

\section{Flow}

In this section we will describe how to flow continuously
from a tridisk-contained marked packing $P(0)$ with contact graph
$G$ to a tridisk-contained packing $P(1)$ with 
contact graph
$H$ (and the same marking). 
As above, the graph $H$ will be obtained from $G$ by removing
an edge $e^-$, (without loss of generality , its last edge)
resulting in the graph $G^-$, and then adding
in its cross edge, $e^+$. For $t\in(0..1)$, the packing $P(t)$
will have contact graph $G^-$. Throughout the flow, the 
outer three disks will remain completely pinned, and thus
we will work with $(\p,\r) \in \RR^{3n-9}$.

We will find this flow as the solution to an ordinary
differential equation that we set up now.
Given any pinned disk configuration 
$(\p,\r)$, 
we have a pinned Jacobian matrix $J_p$. Where $J_p$ is non-singular, we define
the velocity field:
$(\p',\r') := J_p^{-1}[0, 0, ..., 0, 1]^t$. Differentially, this leaves all of the
inversive distances fixed except for increasing the inversive distance
corresponding to $e^-$.
The matrix $J_p$ is non-singular 
for some $(\p,\r)$ (Lemma~\ref{lem:nonsing3}) thus it is
non-singular
over a Zariski-open subset
$U$ of $\RR^{3n-9}$.
This defines a smooth (never zero)
velocity field over $U$
that includes $P(0)$,
setting up 
a system of 
ODEs (ordinary differential equations).

Given this system of ODEs, we start a
trajectory at $P(0)$ and integrate forward in time.
As a smooth ODE, this defines a unique maximal trajectory 
$(\p(t),\r(t))$ 
for some (possibly bi-infinite) open time interval.

The following is standard. See \cite[Section 4.1]{ODE} for a readable discussion.
\begin{theorem}
\label{thm:ode}
Going forward in time, 
a maximal trajectory of a $C^1$
system of 
ODEs defined over an
open set $U$ must leave any compact set contained 
in $U$, or go on for infinite time.
\end{theorem}

\begin{lemma}
\label{lem:SinU}
$\overline{S_{G^-}}\in U$.  So too is $P(0)$.
\end{lemma}
\begin{proof}
From Lemma~\ref{lem:frontcontact}, all disk configurations in $\overline{S_{G^-}}$
are tridisk-contained almost-triangulated packings
with packing graph $G^-$, or triangulated packings with
contact graph $G$ or $H$.
Lemma~\ref{lem:nonsing3} then guarantees they are in $U$.
In particular, it guarantees that  $P(0) \in U$.
\end{proof}

\begin{lemma}
\label{lem:start}
$(\p(t),\r(t)) \in S_{G^-}$ for
sufficiently small and positive $t$.
\end{lemma}
\begin{proof}
  From Lemma~\ref{lem:SinU}, the ODE is
  well defined at $P(0)$.
  The ODE is constructed to maintain  the equality constraints along 
  the edges of $G^-$.   Also by construction,
  the velocity
  of our trajectory at $t=0$ must destroy the contact
  along the edge $e^-$ for $t>0$. For small enough
  $t$, no other inequality constraints can become
  violated.
\end{proof}

\begin{remark}
\label{rem:also1}
This lemma also tells us
that every packing $P$ with contact graph
$G$  
must be the frontier point
of at least one connected component of 
$S_{G^-}$. 
By symmetry, the same is true of 
every packing $Q$ with contact graph
$H$.  
\end{remark}

\begin{lemma}
\label{lem:ifleave}
For any trajectory starting in $S_{G^-}$, if the trajectory leaves $S_{G^-}$,
it must first hit a frontier point of $S_{G^-}$.
\end{lemma}
\begin{proof}
The set $S_{G^-}$ is defined by the intersection of some strict inequality (open) conditions and some equality constraints.
  Our ODE is constructed to maintain  these equality constraints.
  To leave the set it must violate an inequality, and so by continuity, the
  trajectory must first violate this inequality with an equality. Such a point is in the
  closure of $S_{G^-}$ making it a frontier point.
\end{proof}

\begin{lemma}
\label{lem:mustleave}
Any trajectory starting in $S_{G^-}$ 
must leave $S_{G^-}$ at some positive 
$t$.
\end{lemma}
\begin{proof}
From  Theorem~\ref{thm:ode},
(using Lemmas~\ref{lem:bounded} and \ref{lem:SinU}) if the trajectory is only 
defined for a finite time, then 
it must leave
$S_{G^-}$.

But if the trajectory goes on for infinite time, it also must
leave $S_{G^-}$ by the following argument.
Our ODE increases the inversive distance along $e^-$
at a constant rate, and since the radii are bounded
away from zero (Proposition~\ref{prop:nozero}), their centers must be moving
apart at a lower bounded rate.
Meanwhile $S_{G^-}$ is bounded  (Lemma~\ref{lem:bounded}). 
\end{proof}

\begin{proposition}
\label{prop:drive}
Any trajectory starting in $S_{G^-}$ 
 must hit a positive $t$ where 
  the edge $e^+$ is in contact. 
\end{proposition}
\begin{proof}
From Lemma~\ref{lem:mustleave}, the trajectory must leave $S_{G^-}$.
From Lemma~\ref{lem:ifleave}, it must leave at a frontier point.
From Lemma~\ref{lem:frontcontact} this must either be contact along
$e^-$ or $e^+$. But by construction, the inversive distance along $e^-$
is increasing away from $1$, so
we cannot achieve contact along $e^-$.
\end{proof}

\begin{remark}
\label{rem:also2}
This proposition also tells us that every connected
component of $S_{G^-}$ must have a frontier
point $Q$ with contact graph $H$. 
From the uniqueness of ODE solutions, two trajectories cannot cross, and thus
only one component of $S_{G^-}$ can have $Q$ as a frontier point.
By symmetry, every component also must have a frontier 
point $P$ with contact graph $G$, and not other component can hit $P$.

In summary, each connected component of
$S_{G^-}$ must connect one tridisk-contained packing with contact graph  $G$ to one 
tridisk-contained packing 
with contact graph $H$.
And each tridisk-contained packing 
with contact graph $G$ (resp. $H$)
is an endpoint of exactly one component
of $S_{G^-}$.
\end{remark}

We will stop our flow when the edge
$e^+$ is in contact. 
After a time rescaling, we will consider this $t=1$.
We
denote this 
tridisk-contained
packing with contact graph $H$
as $P(1)$. We refer to this trajectory between $t=0$ and $t=1$ as \defn{our flow}.

With Lemma~\ref{lem:fix}, 
Proposition~\ref{prop:drive},  
and our defined flow we have just proven the following
result:
\begin{theorem}
\label{thm:flow}
Let $P$ be a triangulated packing, with contact graph $G$. Let $e^-$
be a flippable edge of $G$. Let $G^-:=G\backslash e^-$, and let $H$ be the graph
resulting from flipping $e^-$. Then there is a continuous path of packings
$P(t)$, with $t \in [0...1]$ such that:
\begin{itemize}
    \item $P(0)=P$.
    \item For $t$ in the interval $(0,1)$, the contact graph of $P(t)$ is $G^-$.
    \item The contact graph of $P(1)$ is $H$.
\end{itemize}
\end{theorem}

Now let us consider the inversive distance along 
$e^+$ during this flow.

\begin{lemma}
\label{lem:mono}
The inversive distance along $e^+$
    monotonically decreases, with derivative bounded away from zero
    during the flow from $P(0)$
    to $P(1)$.
\end{lemma}
\begin{proof}
Let $f^+$ be the inversive distance along $e^+$.
  If $f^+$ were to change sign, we must have
  a time during our flow when $df^+/dt=0$. 
  As $(\p',\r')\neq 0$ (the velocity field is never zero), 
  this would 
  have to be a time where the Jacobian of the "dual flip",
  defined using the edge of
  $H$, goes singular. But Lemma~\ref{lem:nonsing3}  can be applied just us well
  to this dual flip,
  so this cannot happen. Thus during our flow,  $f^+$ changes monotonically, with
  non-zero derivative.
  Meanwhile, we know that $f^+$ must overall decrease from its starting
  value (non-contact) down to $1$ (contact). 
  Since we are working over a compact time interval $[0..1]$, this derivative is bounded away from zero.
 \end{proof}

\begin{corollary}
\label{cor:backwards}
If we start with the graph $H$ and remove
the edge $e^+$, and set up a \defn{dual flow} starting from $P(1)$,
increasing the inversive distance of this edge,
then this dual flow must drive $e^-$ to contact  and give us $P(0)$. This dual flow must travel along the reverse path of the original
forward
flow, defined using the graph $G$ and the edge $e^-$.
\end{corollary}
\begin{proof}
We can set up a dual flow starting at $P(1)$ and increasing
the inversive distance on $e^+$. From 
Theorem~\ref{thm:ode}, Lemma~\ref{lem:start} and 
Proposition~\ref{prop:drive} (applied to this dual differential 
equation)
there is also a
\emph{unique} trajectory satisfying this  dual
differential equation.
But from 
Lemma~\ref{lem:mono}, the reverse path from $P(1)$
to $P(0)$, 
satisfies this dual differential equation (up to a 
a smooth regular parameterization of the time variable).
\end{proof}

\section{Proving KAT}

With these pieces all in place, we can now establish all of the
elements of the KAT circle packing theorem.

\begin{theorem}
\label{thm:kat1}
Let $M$ be a maximal planar graph on $n$ vertices. 
Then there exists a
 packing of $n$ disks with contact
graph $M$. 
\end{theorem}
\begin{proof}
Let $P$ be a trilaterated packing with $n$ disks and with contact graph $C$. 
From Theorem~\ref{thm:flip} we can combinatorially transform
$C$ to $M$ using a finite sequence of edge flips.
Starting with $P$,
we wish to to flow the packing
across each of the edge flips in our sequence.

To this end, let $G$ and $H$ represent two maximal planar graphs
related by flipping some edge $e^-$. 
Let $P$ be a packing with contact graph $G$.
We mark three disks in $P$ corresponding to a triangle in $G^-$.
Using Lemma~\ref{lem:fix},
we apply a M{\"o}bius transformation to $P$ so that the resulting
marked packing is tridisk-contained. Now we can apply
Proposition~\ref{prop:drive} to flow this packing to a 
new tridisk-contained marked packing $Q$ with contact graph $H$.

We can do this for each edge flip in our sequence, 
arriving at 
a packing with contact graph $M$.
\end{proof}

\begin{theorem}
\label{thm:kat2}
Let $M$ be a maximal planar graph on $n$ vertices. 
Then its  packing
is unique up to a generalized M{\"o}bius transformation.
\end{theorem}
\begin{proof}
We will follow the same sequence of edge flips and flows
of the proof of Theorem~\ref{thm:kat1}.

From Lemma~\ref{lem:canunq}, the trilaterated packing 
is unique up to generalized M{\"o}bius transformations.

Let us look at a single edge flip, flowing  from a tridisk-contained
marked packing
$P$ with contact graph $G$ to 
the  tridisk-contained
packing $Q$  with the same marking and with contact graph $H$.
Let us assume (as an inductive invariant)
 that $P$ was the only 
packing with contact graph $G$ up to a 
generalized M{\"o}bius transformation.  Then from Lemma~\ref{lem:fix2}, 
it is 
the only tridisk-contained packing with this 
contact graph and marking.

Now let $Q'$ be any 
tridisk-contained packing with this marking and contact graph
$H$.
Then from Proposition~\ref{prop:drive}, 
the dual differential equation starting from 
$Q'$  must
flow
to \emph{some} tridisk-contained marked packing $P'$
with contact
graph $G$. 
Setting up the double dual ODE, (which is just our original ODE),
starting from $P'$
and
using
Corollary~\ref{cor:backwards}, the double dual flow must lead us back 
to the marked packing $Q'$. 
But from our assumed marked uniqueness, $P=P'$. And since our double dual ODE is just
our original ODE, then $Q=Q'$ (as marked packings).
Thus from Lemma~\ref{lem:fix2}, $Q$ is the only packing with
contact graph $H$ up to a 
generalized M{\"o}bius transformation.
(See also Remark~\ref{rem:also2} for the geometric
point of view.)

In summary, 
the number of distinct packings, up to 
generalized M{\"o}bius transformations, cannot  increase across
an edge flip, maintaining our inductive invariant, and proving the theorem.
\end{proof}

\begin{remark}
If we fix $G$, but
extend our configuration space 
to allow non-edge disk
pairs to overlap, 
and thus allowing $(G,\p)$ to be
a non-embedding, then all bets (such as 
Lemma~\ref{lem:nonsing3})
are off. Indeed, 
 global uniqueness
no longer holds. (For example, smallest disk in
Figure~\ref{fig:canon} can be changed to be 
identical to the disk labeled $1$ in the figure,
while maintaining all of the necessary 
external tangential contact.)
This leads to especially interesting 
situations when using general inversive distance constraints
on the sphere~\cite{ma}.
\end{remark}

\begin{theorem}
\label{thm:kat3}
Let $G$ be a planar graph on $n$ vertices. 
Then there exists a  packing of $n$ disks with contact
graph $G$.
\end{theorem}
\begin{proof}
We can always \defn{add} edges  to $G$ 
to obtain a maximal planar graph
$M$. From Theorem~\ref{thm:kat1}, there must be a 
tridisk-contained (using Lemma~\ref{lem:fix})
packing $P$ with contact graph $M$. To remove the extraneous contacts,
all we need to do is
create a short lived flow starting at $P$ that increases the
inversive distances along the added edges, while maintaining
contact along the edges of $G$.

Note that in this setting, we may also need to separate the three
outer disks of $P$ (as there may be no triangles at all in $G$).
To do this we will work in the $\RR^{3n-6}$ dimensional
configuration space, where the three outer radii, but not their 
centers, are variable.

From Lemma~\ref{lem:nonsing2}, the center-pinned Jacobian $J_c$,
defined using the edges of $M$ is non-singular at $P$.
From upper semi-continuity of rank, it remains non-singular
in a sufficiently small neighborhood $U$ of $P$.

We define our velocity field as
$(\p',\r') := J_c^{-1}\f'$. 
Where $f'_{ij}>0$ for the added edges, and
$f'_{ij}=0$ for the edges of $G$.

Finally we integrate this flow from $t =0$
to $t=\epsilon$
for a sufficiently small positive $\epsilon$.
\end{proof}

The  KAT circle packing theorem is simply the union of Theorems
\ref{thm:kat1}, \ref{thm:kat2},
and
\ref{thm:kat3}.

\section{Generalizations}

\subsection{Other Powerful Tools}
\label{sec:guo}
This proof the KAT circle packing theorem is based on establishing
the infinitesimal rigidity of the inversive distances along the edges
of a maximal planar graph $G$, for any $(\p,\r)$ that is 
an almost-triangulated packing with contact graph $G^-$.
We established this in Section~\ref{sec:jacob} using
the equilibrium condition.
In particular, we use the fact that
all but one of the edges of $G$ represent
tangential contact.

Another set of tools for proving such infinitesimal rigidity
results come from the work of Guo~\cite{guo}. Although, he
works in the intrinsic setting of polyhedral metrics, his
result can also be applied to any set of disks $(\p,\r)$,
with any inversive distances in $[0..\infty]$ along the
edges of $G$,
whenever
the extrinsic $\RR^2$ drawing, $(G,\p)$, is a planar embedding.

Now in our setting, we know that $(G^-,\p)$ is a planar embedding,
but we still have one pesky edge $e^-$, which might become out of order during the flow.
And at some time, one of the triangles with this edge can become degenerate.
Luckily, it seems  that 
one can always find a M{\"o}bius transformation (that does not invert any disks)
under which $(G,\p)$ does become an embedding. 
The \emph{existence} of such an embedded drawing should be
enough to argue the non-singularity of $J_p$ for the packing
under consideration.

\subsection{More General Inversive Distance Targets}

The full KAT theorem 
covers
configurations
of disks that are allowed some amount of overlap~\cite{thurston,bowst}.
One fixes each of the inversive distances
along the edges $ij$ of a maximal planar graph
$G$ to some target value $f_{ij}$ in the interval $[0,1]$
(instead of equal to $1$).
This allows the disks across an edge to intersect with some 
fixed angle $\le \pi/2$.
Disks not connected by an edge are still not allowed to overlap.
In this setting, for existence,
one must also add in an assumption about the sums of the overlapping angles
in $3$-cycles that are not triangles in $G$, and also on the sums over $4$-cycles.
It would be interesting to see if we can use our 
flow based approach to start with a  packing with tangential contact along 
the edges of $G$,  and then 
flow the disks together to overlap by the desired amount.
For this, it might be useful to apply the tools of Section~\ref{sec:guo}.

Work starting with Rivin~\cite{rivin} 
has investigated  the situation with 
$f_{ij} \in (-1,1]$, (allowing disk intersections up to $\pi$). 
Existence and uniqueness can be shown  under some additional assumptions. 
Again, it would be interesting to see if we can apply our 
flow based approach to this setting.

If one has a set of disks $(\p,\r)$, with some set of inversive distances
$f_{ij} \in [0,\infty]$ along the edges of a maximal planar graph $G$,
the story gets even more complicated.
If we work in the category of disks where $(G,\p)$ is a planar embedding, then 
infinitesimal rigidity follows from~\cite{guo,xu} and global uniqueness
follows from~\cite{luo,xu} (though see~\cite{ma,bowma}
for non-uniqueness in the spherical embedded 
setting).
But when we place such general
inversive distance targets on the edges of $G$, then existence
conditions 
are not well understood. 

\subsection{Higher Genus}
The KAT theorem also generalizes to higher genus settings~\cite{thurston}.
We conjecture that given a graph $G$ that triangulates the torus, and
a packing on a flat torus with this contact graph, we can flip an edge,
and smoothly flow the packing across this flip (this necessarily requires us
to also flow the shape of the torus). For example, in his book, "Regular Figures" \cite{LFT}
L\'aszl\'o Fejes T\'{o}th takes the periodic packing on the left in Figure \ref{fig:torus-flip}, which is triangulated, and continuously deforms the packing.
This is a flip-and-flow motion as we have described in this paper, 
applied to the edge between the blue and yellow regions,   
but now on a torus.
This flow changes the 
radii as well as the density of the packing.  If one continues the motion, the packing eventually moves to the standard packing with all radii the same as in the right of the Figure.

We would also be interested in seeing if our method extends to surfaces of higher genus in the hyperbolic setting.

\begin{figure}[h]
    \centering
    \includegraphics[width=0.8\textwidth]{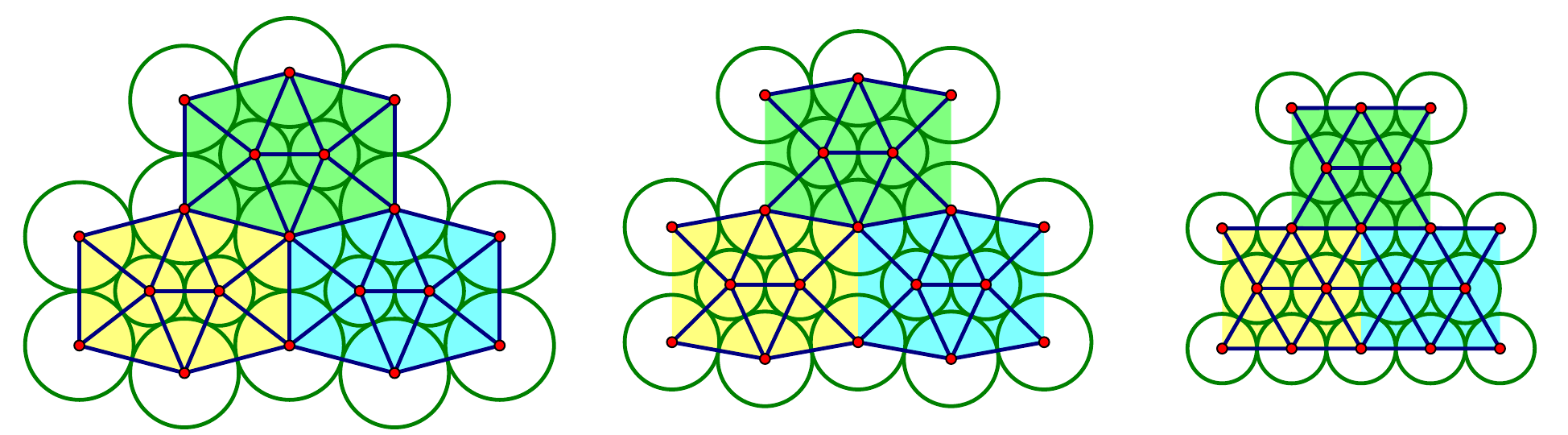}
    \caption{This shows a continuous flip of an edge of a triangulated packing of a torus as used by L\'aszl\'o Fejes T\'{o}th.  The packing on the left is the start of the continuous motion of the packing.  The middle packing was used because the relative ratio of the packing radii were closer to one, and the packing on the right is the standard triangulated packing with all the radii the same.  Three fundamental regions are shown, each in a different color.}
    \label{fig:torus-flip}
\end{figure}

\subsection*{Acknowledgements}
The authors would like to thank Louis Theran for many
helpful discussions along the way.

\bibliographystyle{abbrvnat}
\bibliography{flow}
\end{document}